\documentclass[arXiv]{elsarticle}
\pdfoutput=1
\usepackage{hyperref}
\usepackage{amsmath}%
\usepackage{amsfonts}
\usepackage{xcolor}
\usepackage[algosection,boxed]{algorithm2e}

\graphicspath{{./figs/}}
\DeclareGraphicsExtensions{.pdf,.png,.jpg,.eps}
\DeclareMathOperator*{\argmax}{arg\,max}
\newtheorem{theorem}{Theorem}[section]
\newtheorem{lemma}{Lemma}
\newtheorem{assumption}{Assumption}[section]
\newtheorem{problem}{Problem}[section]
\newtheorem{definition}{Definition}[section]
\newtheorem{remark}{Remark}[section]

\newcommand{\N}{\mathbb{N}}
\newcommand{\R}{\mathbb{R}}

\makeatletter
\def\ps@pprintTitle{%
   \let\@oddhead\@empty
   \let\@evenhead\@empty
   \let\@oddfoot\@empty
   \let\@evenfoot\@oddfoot
}
\makeatother

\begin{document}

\begin{frontmatter}

\title{Optimal and sustainable management of a soilborne banana pest}

\author[ngoa,inria]{Isra\"el {Tankam-Chedjou}\corref{mycorrespondingauthor}}
\cortext[mycorrespondingauthor]{Corresponding author}
\ead{israeltankam@gmail.com}
\author[inria]{Fr\'ed\'eric Grognard}
\author[ngoa,polytech]{Jean Jules Tewa}
\author[inria,inra]{Suzanne Touzeau}

\address[ngoa]{Department of Mathematics, University of Yaound\'e I, Cameroon}
\address[inria]{Universit\'e C\^ote d'Azur, Inria, INRAE, CNRS, Sorbonne Universit\'e, BIOCORE, France}
\address[inra]{Universit\'e C\^ote d'Azur, INRAE, CNRS, ISA, France}
\address[polytech]{National Advanced School of Engineering, University of Yaound\'e I, Cameroon}

\begin{abstract}
In this paper we propose an eco-friendly optimization of banana or plantain yield by the control of the pest burrowing nematode \textit{Radopholus similis}. This control relies on fallow deployment, with greater respect for the environment than chemical methods. The optimization is based on a multi-seasonal model in which fallow periods follow cropping seasons. The aim is to find the best way, in terms of profit, to allocate the durations of fallow periods between the cropping seasons, over a fixed time horizon spanning several seasons. The existence of an optimal allocation is proven and an adaptive random search algorithm is proposed to solve the optimization problem. For a relatively long time horizon, deploying one season less than the maximum possible number of cropping seasons allows to increase the fallow period durations and results in a better multi-seasonal profit. For regular fallow durations, the profit is lower than the optimal solution, but the final soil infestation is also lower.
\end{abstract}


\end{frontmatter}

\section{Introduction}
Crop pests attack cultivated plants or stored crops, causing serious economic damage to the detriment of farmers and threatening food security \cite{Fiala92,Agrios88}. Crop losses to pests are difficult to quantify. Nonetheless, it was estimated that 20 \% to 30 \% of major crop yields were lost because of pests, principally in food-deficit areas \cite{Savary19}, representing 2,000 billion dollars per year \cite{Pimentel09}. Pesticides are still widely used in agriculture: in 2009, almost $3 \times 10^9$ kg of pesticides were used throughout the world, at a cost of nearly 40 billion dollars \cite{Hussain09}. Yet, risks associated with pesticide use have surpassed their beneficial effects, as pesticides have drastic effects on non-target species and hence affect biodiversity, aquatic as well as terrestrial food webs and ecosystems \cite{Mahmood16}. Therefore, the problem of pest control has necessarily to be addressed in an integrated manner, which has motivated the development of alternative environmentally agricultural practices \cite{Nickel73}.

Soilborne pests have a prominent place among plant pathogens. They include fungi, oomycetes, viruses (carried by nematodes or other organisms) and parasitic plants, but also and above all nematodes which promote the infestation of plants by other pathogens \cite{Back02}. The burrowing nematode, \textit{Radopholus similis} [(Cobb, 1893) Thorne, 1949] is a migratory plant parasitic nematode that attacks over 1200 plant species on which it causes severe economic losses in yields \cite{Sarah99,Duncan05} Including banana. Banana is a major staple crop in the tropics and subtropics, and one of the most popular fruits in the world \cite{Fao17, Plowright13}. Collectively called banana, banana and cooking banana, usually named "plantain", are grown in more than 135 countries and are found in most tropical and subtropical regions of the world. There are non-seasonal crops that provide a source of food all year round, making them vital for nutrition and food security. In addition to its export value, banana plantations and small plantain farms are an important source of employment \cite{Temple96,Okolle09}. In Ivory Coast, some studies reported that \textit{R. Similis} was causing average yield reductions of 80\% in banana plantations \cite{vilardebo74}. Additionally, high overall yield losses of 60\% on plantain production in Cameroon were recorded \cite{Fogain00}. When \textit{R. Similis} attacks banana plants, the risks of toppling or heavy over-infestation of the plants are so high that they often lead to stopping growing bananas \cite{Nkendah03,Coyne18}, which impacts on farmers' returns. Fighting this pest therefore represents a major challenge in tropical areas. 
However, as with other plant pathogens, the control of \textit{R. similis} still requires pesticides which are not always very effective and pollute the environment \cite{Gowen97}. Fortunately, some alternative practices such as flooding and fallowing are carried out to reduce the impact of soilborne pathogens like \textit{R. similis} \cite{Chabrier03}. Fallows in particular appear to be a sustainable control method, as nematodes undergo a fast decay in the soil \cite{Chabrier03,Chabrier05,Chabrier10}. In this paper, we aim to assess and optimize the efficiency of the deployment of fallow via a mathematical model.

Mathematical modelling and computer simulation are becoming major tools for the study of the evolution of plant epidemics and optimization of pest control. Concerning soilborne pathogens, Gilligan \cite{Gilligan90}, followed by Gilligan and Kleczkowski \cite{Gilligan97} have proposed some mathematical models in the 90s. Madden and Van den Bosch \cite{Madden02} and Mailleret et al. \cite{Mailleret11} introduced the semi-discrete formalism in soilborne pathogen models to obtain multi-seasonal dynamics of crop-pathogen interactions. Following these works, we proposed a multi-seasonal model for the dynamics of banana or plantain crops in interaction with the burrowing nematode \textit{R. similis}, with fallow periods following cropping seasons \cite{TankamChedjou20}. The nematode basic reproduction number was computed for this model, and we showed that for fallow periods longer than a certain threshold, the pest population declines. However, always deploying longer fallows may not be optimal in terms of yield. Indeed, on a finite time horizon, longer fallow periods may imply less cropping seasons. The aim of this paper is therefore to optimize the duration of fallow periods in order to maximize the profit of banana crop on a fixed time horizon. 

Optimization models, in which a host plant cropping is alternated with either an off-season, a non-host cropping or a poor host, exist in the mathematical modelling literature. Van den Berg et al.\ \cite{Berg06,Berg12}, for instance, rely on an extended Ricker model to optimize potato yield losses due to the potato cyst nematode by rotating different potato cultivars.   Taylor and Rodr\`iguez-K\'abana  optimize the economical yield of peanut crops by rotating peanuts (good host) and cotton (bad host) in order to control the peanut root-knot nematode \textit{Meloidogyne arenaria} \cite{Taylor99}. Nilusmas et al. provide optimal rotation strategies between susceptible and resistant crops to control root-knot nematodes and maximize tomato crop yields \cite{Nilusmas20}. Van den Berg and Rossing design optimal rotation strategies between host and non-host crops or fallows over several yearly cycles in order to manage crop losses due to the root lesion nematode \textit{Pratylenchus penetrans}, based on a fairly generic model  \cite{Berg05}. Strategies consist in deciding whether or not to deploy a one-year fallow. It highly differs from our approach, which in a non-seasonal context aims at optimizing the fallow durations on a given time horizon.

 In Section~\ref{Section_Modelling} we describe the model and its parameter, and we define the yield and the profit. In Section~\ref{Section_Optimization} we state the optimization problem and describe the optimization algorithm. In Section~\ref{Section_Results} we provide the solution of the optimization problem on short and long time horizons. We also seek more regular solutions by bounding fallow durations, penalizing extreme durations or setting constant durations. In Section~\ref{Section_discussion} we discuss our results and present possible extensions.

\section{Modelling}\label{Section_Modelling}

\subsection{Plant-nematode interaction model}

Banana is a perennial herbaceous plant widely cultivated in the tropical and subtropical regions. As a non-seasonal crop, bananas are available fresh year-round. It is perennial because it produces succeeding generations of crops. The plant propagates itself by producing suckers which are outgrowths of the vegetative buds set on the rhizome during leaf formation and which share their parent rhizome during their  early development \cite{Eckstein99}. Hence, infested parent plants lead to infested suckers \cite{Eckstein99,Duncan90}. In order to avoid this direct transmission of pests, an alternative reproduction method can be proposed: after the harvest of the bunch, the banana plant is uprooted and a new healthy vitro-plant is planted, usually after the fallow \cite{Chabrier03}. The young sucker produces roots continuously until the flowering \cite{Beugnon66}; then absorbed nutrients are essentially used for the growth of the fruit bunch. 

The nematode \textit{Radopholus similis} is an obligate parasite that feeds on banana roots. It penetrates the banana root, travels and feeds on the root cortex. \textit{R. similis} directly destroys cells and also facilitates the entry and development of saprophagous and secondary parasite \cite{Blake66,Loridat89} inducing root necrosis \cite{Hugon88}. It mostly reproduces sexually, even though females can use parthenogenesis if males are lacking \cite{Kaplan00}. Fertilized females lay about five eggs daily over their gravidity period which can last 2 weeks \cite{Marin98}. From these eggs, young larvae emerge, which can either remain in the root or end up in the soil in search of new roots to colonize \cite{Hugon88}. In general, when hosts are present, very few of \textit{R. similis} are found in the soil whereas higher densities are found in roots and rhizomes \cite{Araya95}. 

The model we study here is based on previous works \cite{Tankam18,TankamChedjou20}. It considers a compartment $P$ for the population of free nematodes in the soil, a compartment $X$ for the population of infesting nematodes in the roots and a compartment $S$ for fresh roots biomass in grams. We name $t_k$ the starting point of the $(k + 1)$-th season and we set $t_0 = 0$ the starting point of the first season. We consider an initial infestation $P(0^+) = P_0 \ge 0$, and assume that the new suckers planted at the beginning of each cropping season have the same root biomass and are nematode-free, such that $S(t_k^+)=S_0$ and $X(t_k^+) =0$ for all $k\ge0$; the superscript ``+'' stands for the instant that directly follows. 

We work at the scale of a single plant. The dynamics of the interaction between nematodes and plant roots during the cropping seasons is given by the following equation for $t \in (t_k,t_k+D]$: 

\begin{equation} \label{Cropping1_simple}
    \left\{\begin{aligned}
     \dot{P} &= -\beta P S + \alpha a (1 - \gamma) \frac{S X}{S + \Delta} - \omega P,\\
     \dot{S} &= \rho(t) S \Big(1-\frac{S}{K}\Big) - a \frac{ SX}{S + \Delta}, \\
     \dot{X} &= \beta P S + \alpha a \gamma \frac{S X}{S + \Delta} - \mu X;
\end{aligned}\right.
\end{equation}
where $\beta$ is the infestation rate of free nematodes ($P$), $a$ is the consumption rate of infesting nematodes ($X$) on fresh roots ($S$), $\alpha$ is the conversion rate of ingested roots, $\gamma$ is a proportion of nematodes laid inside, $\mu$ and $\omega$ are mortality rates, $\Delta$ is the half-saturation constant, $\rho(t)$ is the logistical growth of roots. If we name $d$ the duration between the beginning of the cropping season and the flowering of the plant, and $D$ the duration between the flowering and the harvest, then $\rho(t)$ takes the form:

    \begin{equation*}
     \rho(t) =
    \begin{cases}
    \rho &\text{for } t \in (t_n,t_n + d], \\
    0  &\text{for } t \in (t_n + d, t_n + D].
\end{cases} 
\end{equation*}

In the following, we will term, when needed, the dynamics of (\ref{Cropping1_simple}) during the $(t_k,t_k+d]$ intervals ``the first subsystem of (\ref{Cropping1_simple})'' while ``the second subsystem of (\ref{Cropping1_simple})'' will concern  $(t_k+d,t_k+D]$, with $\rho=0$. 

At the end of a cropping season, i.e.\ at $t = t_k+D$, the plant is uprooted and the uprooting cannot be perfect. Hence we assume that a fraction $q$ of infesting nematodes remains in the soil in addition to the free nematodes inherited from the cropping season. That is traduced by the switching $P(t_{k}+D^+) = P(t_{k}+D) +qX(t_{k}+D)$. 

The only dynamics that remains is the dynamics of free nematodes that undergo an exponential decay \cite{Chabrier10} during a fallow of length $\tau_{k+1}$ until the beginning $t_{k+1}^+$ of the next cropping season. We therefore have the following switching rule between seasons: 

\begin{equation} \label{switch_new_season}
    \left\{\begin{aligned}
    P(t_{k+1}^+) &= \Big( P(t_{k}+D) + q X(t_{k}+D)\Big)e^{-\omega\,\tau_{k+1}},\\
    S(t_{k+1}^+) &= S_0, \\
    X(t_{k+1}^+) &= 0,
\end{aligned}\right.
\end{equation} 

Equations (\ref{Cropping1_simple}) and (\ref{switch_new_season}) form our multi-seasonal model for the dynamics of banana-nematodes interactions with a distribution $(\tau_{k+1})_{k \ge 0}$ of fallow periods. 

The diagram in Figure~\ref{diagramme} summarizes the described multi-seasonal dynamics.

\begin{figure}[htp]
   \centering
    \includegraphics[width=.85\linewidth]{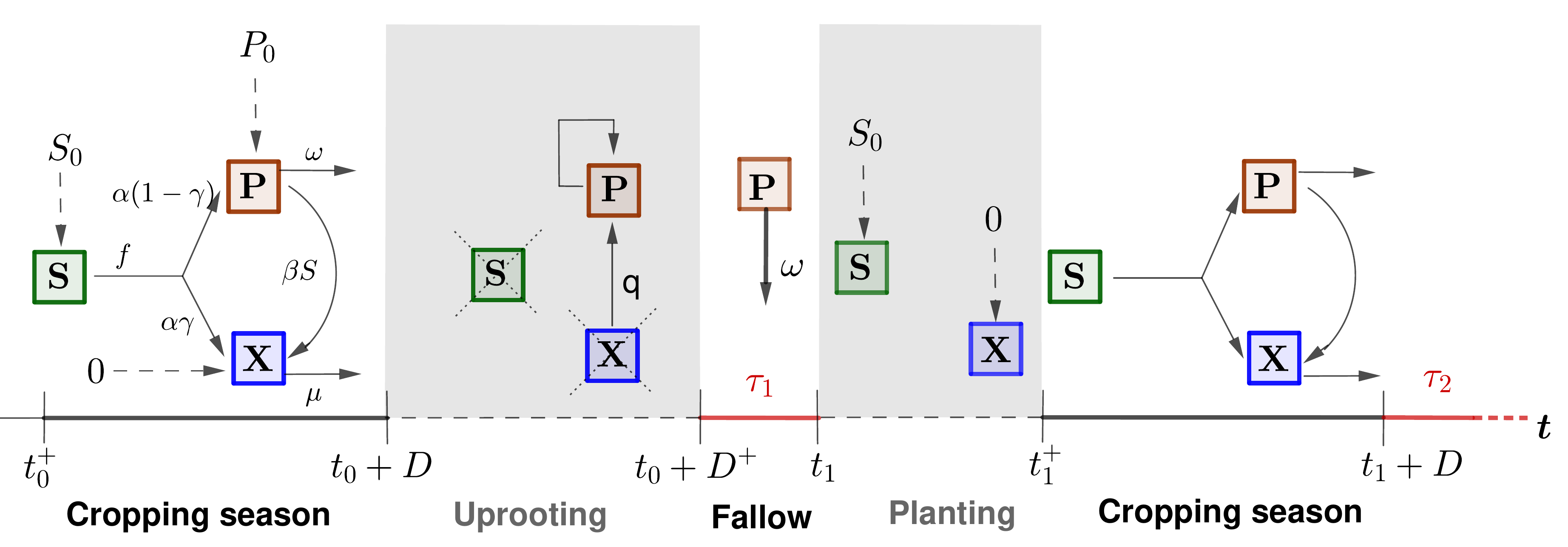}
    \caption{Diagram of system (\ref{Cropping1_simple},\ref{switch_new_season}) during two cropping seasons. Cropping seasons are followed by fallow periods (duration $\tau_i$), during which free pests ($P$) decay exponentially at rate $\omega$. The time runs continuously during the cropping seasons and the fallow periods and is represented by solid lines on the time axis. The discrete phenomena, that are planting and uprooting, are represented by dotted lines. When a new pest-free sucker is planted at $t_i^+$, the fresh root biomass is initialized at a constant weight $S_0$, infesting pests at $0$, while the free pest population remains the same. When the plant is uprooted at $t_i+D^+$, a proportion $q$ of infesting pests turns into free pests. }
    \label{diagramme}
\end{figure}
   
It has been shown that the model (\ref{Cropping1_simple},\ref{switch_new_season}) is well-posed \cite{TankamChedjou20}. We now define what is the economical profit that can emerge from this model.

\subsection{Yield and profit}

Banana roots are responsible of the absorption of nutrients. After the flowering, these nutrients are mainly used for the growth of the banana bunch. If the economic yield of a bunch depends on its weight, then this yield is related to the biomass of fresh roots during the bunch's growth period. We make the hypothesis that a season is profitable only if it is complete. A metric to capture the yield of the $(k+1)$-th cropping season from the model (\ref{Cropping1_simple},\ref{switch_new_season}) can therefore be given by the following formula: 

\begin{equation*}
Y_{k}  = \int_{t_{k}+d}^{t_{k}+D}W(t)\,S(t)\,dt,
\end{equation*}
where $W(t)$ is a weighting function \cite{Hall07}. 

From reference \cite{Serrano05}, the weighting function $W(t)$ appears to be a constant $W(t) = m$. So we can rewrite the previous expression of yield as follows:

\begin{equation}\label{rendement_saison}
Y_{k}  = m \int_{t_{k}+d}^{t_{k}+D}S(t)\,dt.
\end{equation}
 
As each healthy sucker has a cost \cite{Ngo11}, we subtract the cost of a healthy sucker from the yield to obtain the profit:

\begin{equation}\label{profit_saison}
R_{k}  = m \int_{t_{k}+d}^{t_{k}+D}S(t)\,dt -c.
\end{equation}

In equation~(\ref{profit_saison}), the biomass $S(t)$ depends of the infestation. Because of the switching law~(\ref{switch_new_season}), this infestation depends on all the fallow periods that have preceded the current season. Hence, the cumulated profit after the deployment of $n$ fallow periods is the sum of the corresponding $(n+1)$ cropping seasons and depends on the distribution of fallow periods. Its expression is given by:

\begin{equation}\label{profit}
R(\tau_1, \dots, \tau_n) = \sum_{k=0}^{n}R_k
\end{equation}

\subsection{Parameter values}\label{Subsection_param}

We rely on parameters in reference \cite{TankamChedjou20}. Most of them are set to realistic values obtained from experimental studies in the literature. However, some parameters cannot be easily measured, as for instance the consumption rate ($a$) of \textit{Radopholus similis} that is evaluated from the size, and therefore the mass, of a single pest \cite{VanWeerdt60}. Some data also come from different geographic regions, and we assume that they are compatible. For example, we graphically evaluated  the parameter $m$ of the weighting function intervening in the yield equation~\eqref{rendement_saison} based on plantations in Costa Rica, in a publication which relates the yield in boxes per hectare and per year (a box  weighing 18.14~kg) to the functional root weight in grams per plant \cite{Serrano05}. To convert this yield per hectare into the yield per plant, we used plant high density data from plantations in Latin America and the Caribbean \cite{Rosales10}; based on FAO data \cite{Fao17}, we converted the yield into a monetary yield. The currency used is the Central African CFA franc (XAF). 

Table~\ref{Param} summarizes the parameters considered in our optimization. 

\begin{table}
\centering
\small
\begin{tabular}{lll}
  \hline
  Param. & Description & Value \\
  \hline
  $d$ & Root growth duration  & 210 days \\
  $D$ & Cropping season duration & 330 days  \\
  $\beta$ & Infestation rate &  $10^{-1}$ \\
  $K$ &  Maximum root biomass &  150 g\\
  $\rho$ & Root growth rate & $0.025$ day$^{-1}$\\
  $\omega$ & Mortality rate of free pests &  $0.0495$ day$^{-1}$ \\
  $\mu$ & Mortality rate of infesting pests & $0.045$ \\
  $a$ & Consumption rate & $2.10^{-4}$ g.day$^{-1}$\\
  $\alpha$ & Conversion rate of ingested roots & $400$ g$^{-1}$ \\
  $\Delta$ & Half-saturation constant & $60$ g\\
  $\gamma$ & Proportion of pests laid inside & $0.5$ \\
  $q$ & Proportion of pests released in soil after uprooting &  $5\%$\\
  $S_0$ & Initial root biomass &  $60$ g \\
  $P_0$ & Initial soil infestation & $100$  \\
  $m$ & Root to yield conversion rate & $0.3$ XAF.g$^{-1}$.day$^{-1}$ \\
  $c$ & Cost of a banana healthy sucker & $230$ XAF \\
  \hline
\end{tabular}
\caption{Parameter values. The value of parameter $c$ is found in \cite{Ngo11}. The value of parameter $m$ is estimated from references \cite{Serrano05,Rosales10,Fao17}. See \cite{TankamChedjou20} for more details on remaining parameters.} 
\label{Param}
\end{table}

\section{Optimization}\label{Section_Optimization}

From switching rule (\ref{switch_new_season}), long fallow durations $\tau_n$ lead to the reduction of the soil infestation. Equation~(\ref{rendement_saison}) shows that the seasonal yield is linked to fresh root biomass that depends on the infestation level throughout the season. Increasing the fallow durations to drastically reduce the pest population increases the yield. However, on a fixed and finite time horizon that spans several seasons, increasing the fallow durations may reduce the number of cropping seasons and hence the multi-seasonal profit. For example, if we consider $D = 330$~days and a time horizon $T_{max} = 1000$~days, then $3$ cropping seasons can be completed with short fallow periods, for instance of $2$ and $5$~days. However, the crops will be hampered by severe infestations that may reduce the profit. In contrast, if fallow periods are longer, for instance $20$ and $50$~days, the profits of the first two seasons increase, but the third cropping season cannot be completed within $T_{max}$. 

\begin{figure}[htp]
   \centering
    \includegraphics[width=.8\linewidth]{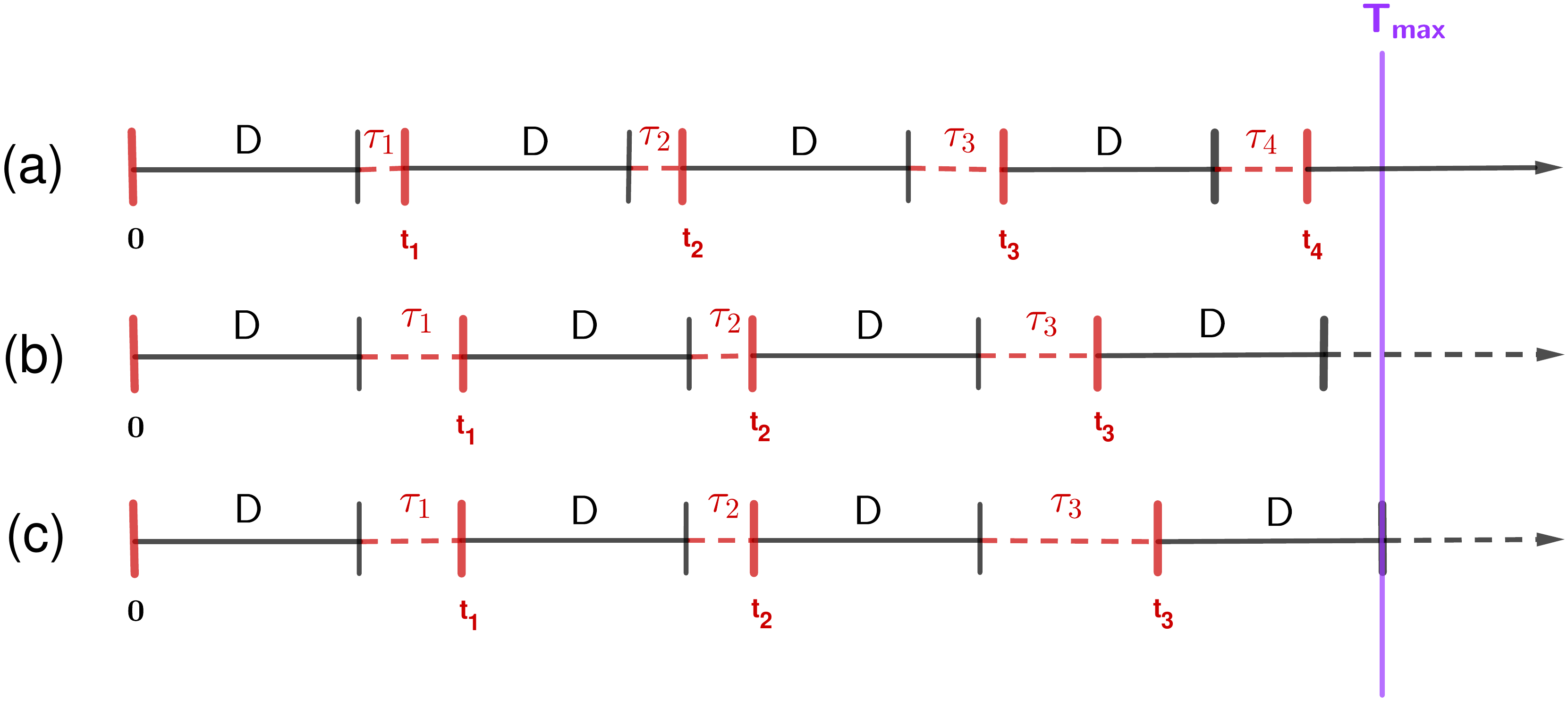}
      \caption{Possible occurrences of $T_{max}$. In (a) $T_{max}$ occurs in the middle of a cropping season, in (b) during a fallow period, and in (c) at the end of a cropping season, at the same time as the harvest.}
    \label{occurence}
\end{figure}  

The optimization problem here is to find a sequence of fallow durations that maximizes the total profit. For the problem to be relevant, the time horizon should allow to deploy at least one fallow period. Still denoting $T_{max}$ the time horizon, it corresponds to the following assumption:

\begin{assumption}\label{assumption_span}
The time horizon spans at least two seasons: $T_{max} > 2D$.
\end{assumption}

We hence state the following problem:

\begin{problem}\label{Problem_simple}
Under assumption \ref{assumption_span}, find a sequence of fallow durations $(\tau_1^*, \dots, \tau_n^*)$ such that maximizes $R$ defined in equation~\eqref{profit}.
\end{problem}

\subsection{Location of the optimal solutions}

 If $T_{max}$ hits the middle of a cropping season (Figure~\ref{occurence}(a)) then this season is useless in terms of profit because its harvest occurs after $T_{max}$. In the same way, if $T_{max}$ hits a fallow period (Figure~\ref{occurence}(b)), then this fallow is useless because it is not followed by a cropping season within $T_{max}$. In both cases, the time elapsed between the last harvest and $T_{max}$ might be added to the last useful fallow such that $T_{max}$ hits the end of a cropping season (Figure~\ref{occurence}(c)). This might increase the yield of the crop. Indeed, according to switching rule (\ref{switch_new_season}), increasing the length of the last fallow period leads to a reduced number of pests at the beginning of the last season. This reduction of pests is supposed to reduce the pest population throughout the season and therefore increase the root biomass and the profit. So the last harvest of the optimal solution should occur at $T_{max}$. But such ideal behaviour only holds in dynamical systems that show a certain ``monotonicity''. Definition \ref{def_monotonicity} describes such monotonicity for system (\ref{Cropping1_simple}).  

\begin{definition}{(Monotonicity)}\label{def_monotonicity}

System (\ref{Cropping1_simple}) is said to be monotone on the interval $(t_k, t_k+D]$ according to initial soil infestation $P(t_k^+) = P(t_k)$ if:
\begin{enumerate}
\item $\tilde P_k \ge P_k \Rightarrow P\big(t;t_k^+,(\tilde P_k,S_0,0)\big) \ge P\big(t;t_k^+,(P_k,S_0,0)\big)$ for all $t \in (t_k, t_k+D]$
\item $\tilde P_k \ge P_k \Rightarrow X\big(t;t_k^+,(\tilde P_k,S_0,0)\big) \ge X\big(t;t_k^+,(P_k,S_0,0)\big)$ for all $t \in (t_k, t_k+D]$
\item $\tilde P_k \ge P_k \Rightarrow S\big(t;t_k^+,(\tilde P_k,S_0,0)\big) \le S\big(t;t_k^+,(P_k,S_0,0)\big)$ for all $t \in (t_k, t_k+D]$
\end{enumerate}
Where $(P,S,X)\big(t;t_k^+,(P_k,S_k,X_k)\big)$ is the solution on $(t_k, t_k+D]$ of equation (\ref{Cropping1_simple}) with initial condition $(P_k,S_k,X_k)$ at $t =t_k^+$.
\end{definition}

We can reformulate the preceding argument as follows. 
Let $t_n$ be the starting point of the last useful cropping season, i.e.\ the last season for which $\delta_n = T_{max}-(t_n+D) \ge 0$. If the last harvest occurs at $T_{max}$, then $\delta_n = 0$. Otherwise, let $\tilde\tau_n = \tau_n + \delta_n$ and let $\tilde t_n = t_{n-1}+\tilde\tau_{n}$.

Because of switching law (\ref{switch_new_season}), we have $P(t_{n}) = P(t_{n-1} + \tau_{n}) \ge P(t_{n-1}+\tilde\tau_{n}) = P(\tilde t_{n})$. As a consequence, in case of monotonicity, we have
$ S(t_n + t) \le S(\tilde t_n + t)$, for  $t \in (0,D]$.
Hence, 

\begin{equation}\label{biomass_gain}
\int_{t_n+d}^{t_n+D}S(t)dt \le \int_{\tilde t_n+d}^{\tilde t_n+D}S(t)dt
\end{equation}
and the profit of season $n$ is higher for fallow duration $\tilde \tau_{n}$. 

We add the following assumption. 
\begin{assumption}\label{assumption_mon}
System (\ref{Cropping1_simple}) is monotone according to definition \ref{def_monotonicity}.
\end{assumption}

\begin{remark}
Assumption \ref{assumption_mon} means that the fewer the initial pests, the lower the infestation throughout the season, and the larger the root biomass. However, for such a predator-prey-like system this property may not hold depending on the parameters values. Indeed, when the level of infestation is high, root biomass $S$ undergoes overconsumption. Such overconsumption induces the decline of root biomass that is food for nematodes. This food decline leads to the decline of nematodes and therefore to the recovery of the root biomass, if the overconsumption occurs early enough during the root growth period $(t_k,t_k+d]$. Such dynamics give rise to cycles that induce the loss of monotonicity. If the pests are ``not too abundant'', this overconsumption scenario should not appear and the monotonicity holds at least for the finite duration $D$.
\end{remark}

We surmise that there exists a reasonable level of infestation below which Assumption~\ref{assumption_mon} is realistic and we illustrate it numerically. With parameters in Table~\ref{Param}, we plot in Figure~\ref{monotonie} the curves of free pests ($P$), infesting pests ($X$) and fresh root biomass ($S$) on a single season, for a large range of infestation values $P(t_k^+)$ at the beginning of cropping season $k$ that encompasses realistic values that are usually below 100. It shows that Assumption~\ref{assumption_mon} holds for variables $P$, $S$ and $X$, for realistic values of $P(t_k^+)$ below 200. Indeed, the curve order is conserved throughout the season (curves do not cross), so the monotonicity condition is verified for $P$, $S$ and $X$.

\begin{figure}[htp]
   \centering
    \includegraphics[width=.75\linewidth]{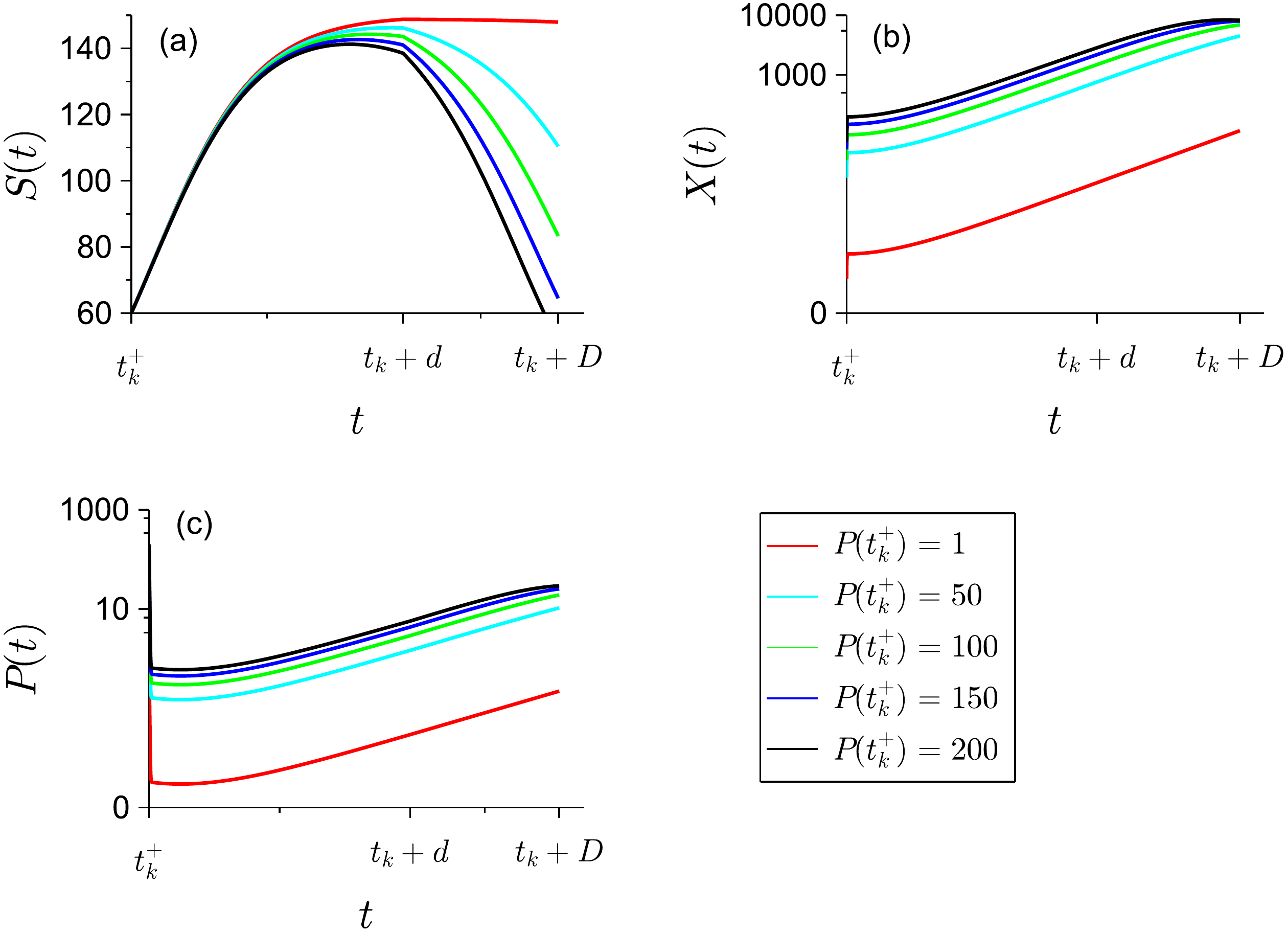}
    \caption{Curves of pests populations ($X$ and $P$) and fresh root biomass ($S$) for different values of the initial infestation $P(t_k^+)$. For initial infestations lower than 200, (a) the greater the initial infestation, the lower the curves of fresh root biomass on the domain $(t_k, t_k + D]$ ; (b) and (c) The greater the initial infestation, the higher the curves of pests on the domain $(t_k, t_k + D]$.}
    \label{monotonie}
\end{figure}

Nevertheless, we could build a counterexample setting two parameters to unrealistic values by observing the system dynamics during the first two seasons. We considered a very high and unrealistic value of the initial infestation $P(t_0^+) = 300$, compared to the reference value $P_0 = 100$. Then by setting the proportion of pests released in the soil after uprooting to $q = 100 \%$, we ensured that the infestation at the beginning of the second cropping season was higher than with the reference value $q = 5\%$. We varied the fallow duration $\tau$. The higher the $\tau$ values, the lower $P(t_1^+)$ at the beginning of the second cropping season. If monotonicity Assumption~\ref{assumption_mon} held, then we would expect a lower $S$ curve for a lower $\tau$ and hence a lower profit. However, for instance for $\tau=2$ and $\tau=10$, that yield to $P(t_1^+) = 6860$ and $P(t_1^+) = 4617$ respectively, it did not hold, as shown in Figure~\ref{contreexemple}: the $S$ curves cross (Figure~\ref{contreexemple}(a)) and, as a consequence, the profit is lower for the higher $\tau$ (Figure~\ref{contreexemple}(b)). More generally, for low values of $\tau$, the profit counter-intuitively decreases with $\tau$ (Figure~\ref{contreexemple}(b)). However, this situation is quite unrealistic, since a proportion $q = 100\%$ simply means that there is no uprooting of the old plant and that all the nematodes remain in the soil. 

\begin{figure}[htp]
   \centering
    \includegraphics[width=.4\linewidth]{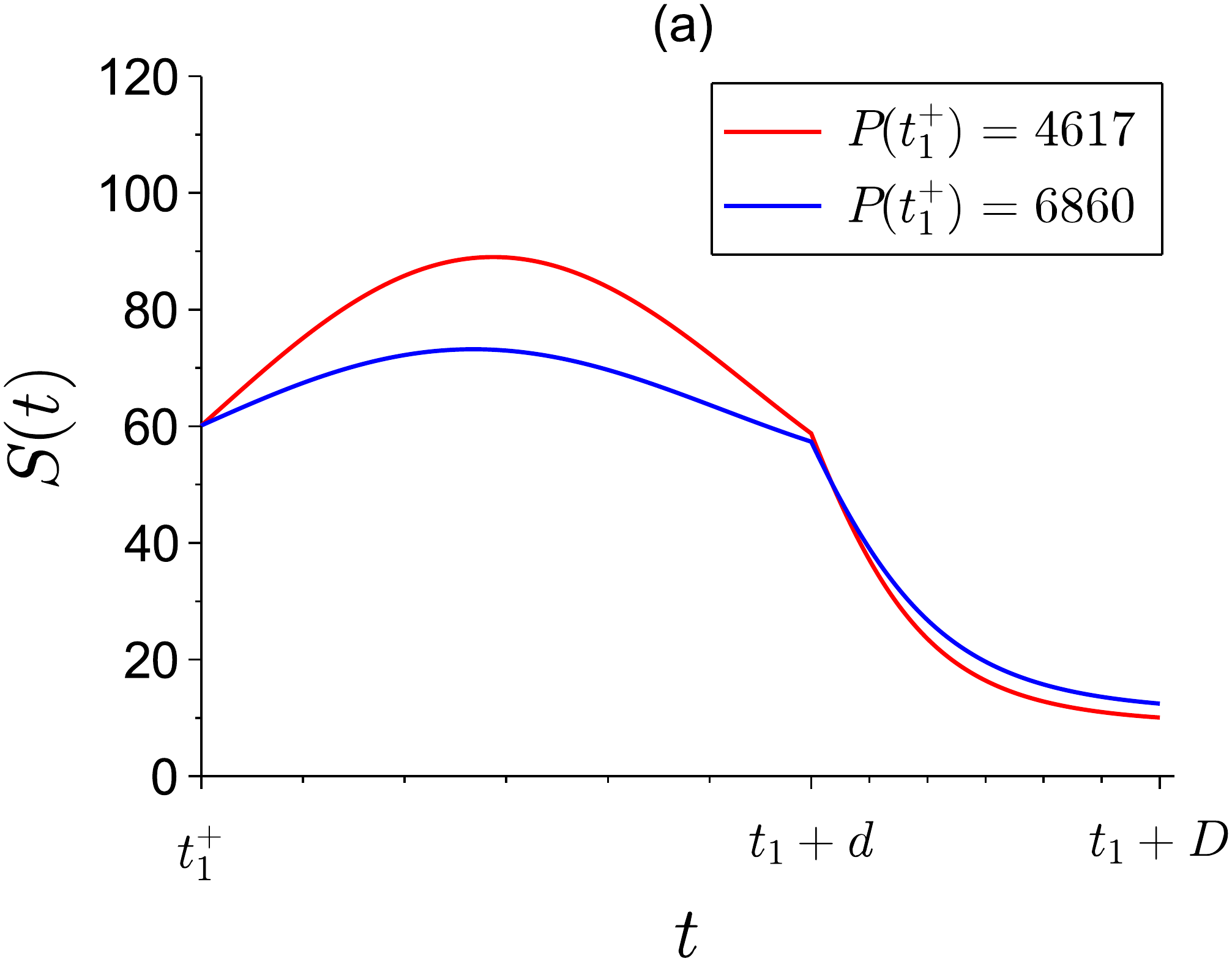}
    \includegraphics[width=.4\linewidth]{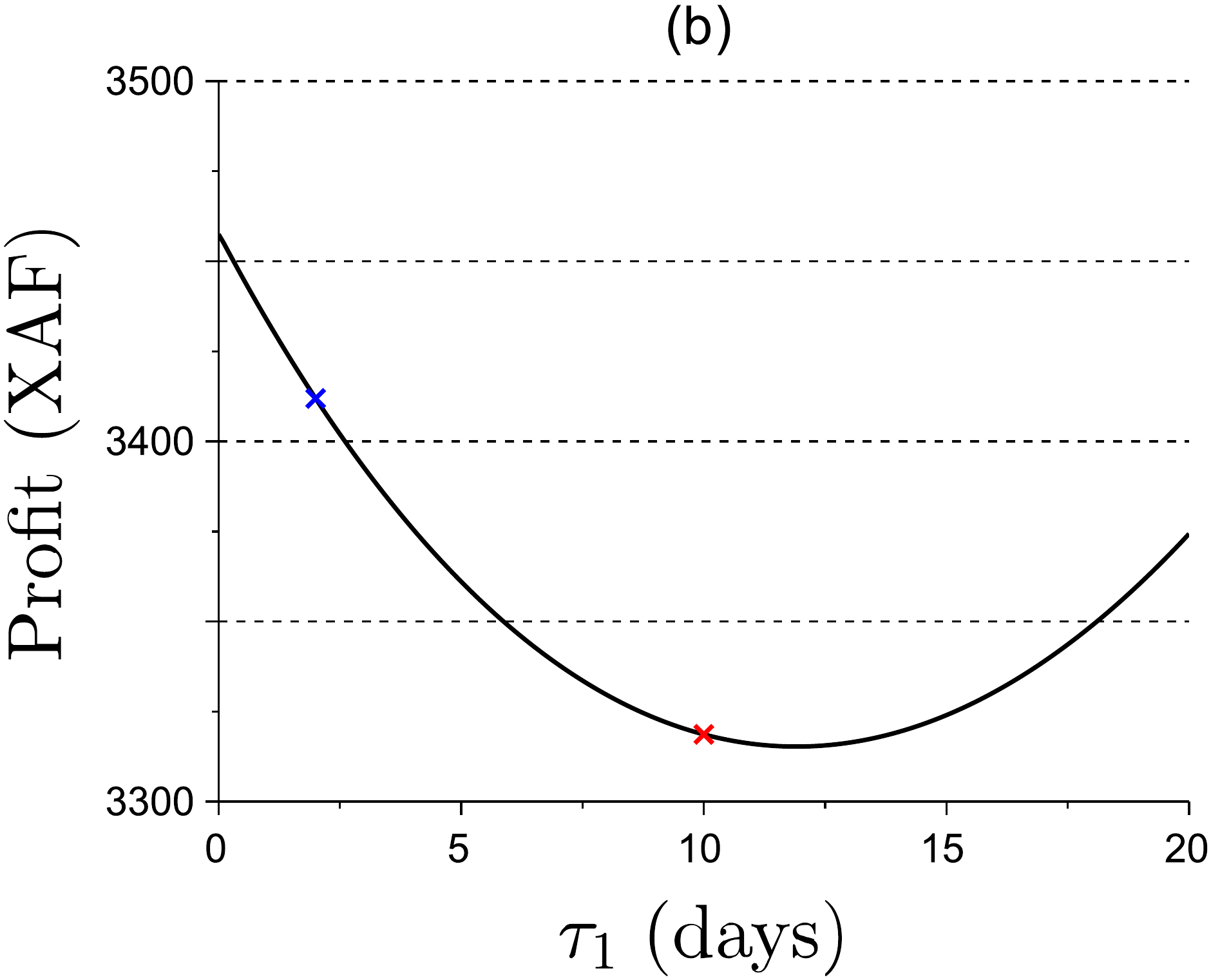}
    \caption{Counterexample: loss of monotonicity when $q = 300$ and $P_0 = 300$. (a) Root biomass during the second cropping season for two different values of $P(t_1^+)$ arising from $\tau_1 = 2$~days (blue curve) and $\tau_1 = 10$~days (red curve). At the beginning, the blue curve is below the red curve, which is consistent with the monotonicity assumption as $P(t_1^+)$ is higher for the blue curve. However, shortly after $t_1+d$, the relative position of the two curves switches and so the monotonicity Assumption~\ref{assumption_mon} is not verified. (b) Profit for the first two seasons as a function of fallow duration $\tau_1$. The fallow duration $\tau_1 = 2$~days (blue cross) yields a better profit than  the fallow duration $\tau_1 = 10$~days (red cross), although the former corresponds to a higher infestation $P(t_1^+)$ than the latter. This is due to the loss of monotonicity.}
    \label{contreexemple}
\end{figure}

According to the arguments above, we can state the following lemma:

\begin{lemma}
  Under monotonicity Assumption~\ref{assumption_mon}, if Problem~\ref{Problem_simple} admits a solution, then it belongs to one of the $n$-simplexes:
  
\begin{equation}\label{equation_simplex}
  A^n = \big\{ (\tau_1, \dots, \tau_n): \sum\tau_k = T_{max}-(n+1)D \big\},
\end{equation}
with $n = 1, ..., n_{max} \equiv \big\lfloor\frac{T_{max}}{D}\big\rfloor -1$. It means that the last harvest needs to occur at $T_{max}$.
\end{lemma}

Problem~\ref{Problem_simple} can be rewritten as: 

\begin{problem}\label{probleme}
Let $\mathcal{A}$ be the reunion of all the $n$-simplexes $A^n$ ($n \in \{1, ..., n_{max} \}$). Under Assumptions \ref{assumption_span} and \ref{assumption_mon}, find the optimal sequence $(\tau_1^*, \dots, \tau_n^*) \in \mathcal{A}$  of fallow durations that maximizes $R$ defined in equation~\eqref{profit}. 
\end{problem}

We prove the existence of solutions to Problem \ref{probleme}.

\begin{theorem}{(Existence of optimal fallow deployments)}

For all $n_{max} \ge 1$, Problem \ref{probleme} has a solution on the collection $\mathcal{A}$ of the $n$-simplexes $A^n$ ($n \in \{1, ..., n_{max} \}$) defined in equation (\ref{equation_simplex}).
\end{theorem} 

\begin{proof} Let us consider from equation (\ref{profit}) the expression $R(\tau_1, \dots, \tau_n) = \sum_{k=0}^{n}R_k$. First, $(n+1)$ is bounded by $\big\lfloor \frac{T_{max}}{D} \big\rfloor$.
Besides,
\begin{itemize}
\item[(i)] $Y_0 = m\int_{t_0+d}^{t_0+D}S(t)dt$ is finite and doesn't depend on any $t_i$.
\item[(ii)] For all $k \ge 1$, the bounds of the integral $Y_k = m \int_{t_k+d}^{t_k+D}S(t)dt$ continuously depends on $(\tau_1, \dots, \tau_{k-1})$ as $t_k = D + \tau_1 + D +  \tau_2 + \dots + D + \tau_{k-1}$.
\end{itemize}

$S(t)$ continuously depends on the initial condition $(P(t_k+d^+),S(t_k+d^+),X(t_k+d^+))$ of the second subsystem of \eqref{Cropping1_simple} that continuously depends on the initial condition $(P(t_k^+),S(t_k^+),X(t_k^+))$. From switching rule (\ref{switch_new_season}), this initial condition continuously depends on $\tau_{k}$.

Hence, the yield $Y_k$ continuously depends on $(\tau_1, \dots, \tau_{k})$ for all $k \ge 1$. 

It follows that $\sum_{k=0}^{n}R_k$ is upper-bounded and lower-bounded. Therefore, $R$ admits a minimum and a maximum on $\mathcal{A}$.
\end{proof}

\begin{remark}
The maximizing sequence might not be unique. If two or more solutions are optimal we would need to define a tie-break rule. For instance, we could prefer (i) a solution with less cropping seasons;  (ii) within solutions with the same number of cropping seasons, the solution closer to the average fallow duration.   
\end{remark}

\begin{remark}\label{3D}
If $T_{max}<3D$, then a maximum of two cropping seasons and one fallow can be deployed. Problem \ref{probleme} admits a unique solution $\tau_1^* = T_{max}-2D$. 
\end{remark}

\subsection{Optimization algorithm}\label{Subsection_algorithm}

For values of $T_{max}$ that are larger than $3D$, the solution could imply two or more fallow periods. In order to numerically solve the optimization problem (\ref{probleme}), we propose an algorithm of adaptive random search as proposed in Walter and Pronzato \cite{Walter97}, that we adapt to simplexes. This method is useful since the function $R$ may have many local maximizers and it is highly desirable to find its global maximizer. The convergence of this kind of algorithms has been proven in the literature \cite{Solis81}. Algorithm \ref{algo} gives the solution of the optimization problem \ref{probleme}. The profits of the maximizers in each dimension are compared to obtain the optimum. 

\begin{algorithm}[h]
  \KwData{$T_{max},\: D$ \\
  $n_{max} := \bigg\lfloor \displaystyle\frac{T_{max}}{D}\bigg\rfloor -1$
  \hspace{2em}  // maximum number of fallow periods that can be deployed on  $[0,T_{max}]$}
  \KwResult{optimal fallow sequence $\vec{\tau}^*$ of size $n^*$}
  $n^*=1$  \hspace{8em} // initialization \\
  $\vec{\tau}^* := \vec{\tau}^{1,*} =T_{max}-2D$  \\
  \For{$n := 2$ to $n_{max}$}{
    $\vec{\tau}^{n,*} :=$ ARS($n$)  \hspace{3em} // $n$-optimal fallow sequence of size $n$\\
    \If{$R(\vec{\tau}^{n,*}) > R(\vec{\tau}^*)$}{
      $n^*= n$ \\
      $\vec{\tau}^* = \vec{\tau}^{n,*}$
    }
  }
  \caption{Optimization algorithm for the numerical resolution of Problem \ref{probleme}. Integer $n$ corresponds to the number of fallow periods that are deployed on interval $[0,T_{max}]$. For each $n\le n_{max}$, the $n$-optimal fallow sequence $\vec{\tau}^{n,*}$ is computed: for $n=1$ the solution is trivial; for $n>1$, the $n$-optimum is computed using an adaptive random search (ARS) algorithm, adapted to simplex $A^n$. The ARS algorithm is detailed in~\ref{ars}. The optimal fallow deployment $\vec{\tau}^*$ corresponds to the $n$-optimum that yields the highest profit $R$.}\label{algo}
\end{algorithm}

\section{Numerical results}\label{Section_Results}

We provide the solution for small values of the time horizon $T_{max}$ in Subsection~\ref{Subsection_small_dimension} and for high values in Subsection~\ref{Subsection_high_dimension}. The latter relies on the optimization algorithm described above in Subsection~\ref{Subsection_algorithm}.

\subsection{Small dimensions}\label{Subsection_small_dimension}

In small dimensions, when $T_{max}<5D$, up to 3 fallow periods can be deployed. The 3-dimension simplex $A^3$, defined in equation \eqref{equation_simplex}, can be represented on a plane. Hence, we can have a good numerical understanding of the location of the optimal solution of problem (\ref{probleme}), by building a graphical representation of the  profit on the simplex and identifying its maximum. We name ``size of the simplex'' the length of each side of the simplex. We use parameter values in Table~\ref{Param}.

We first set $T_{max} = 1400$~days. Up to 4 cropping seasons, corresponding to 3 fallow periods covering $1400-4\times330=80$~days, can hence be deployed. Figure~\ref{Levels} is a representation of the profit, defined in equation~\eqref{profit}, in the 3-simplex of size $80$~days projected on its first coordinates $(\tau_1,\tau_2)$. The duration of the third fallow period is then  $\tau_3=80-(\tau_1+\tau_2)$.  We notice that:
\begin{itemize}
\item The maximum is obtained for 3 fallow periods and is located at the summit $\tau_1=80$~days. This may be because, when there is enough fallow duration to be distributed  (here $80$~days), a long first fallow can lead to drastic pest reduction and hence better profits for the following cropping seasons. 
\item The profit is low near the point $\tau_3=80$~days (that corresponds to $\tau_1=\tau_2=0$) and increases toward the edge $\tau_3=0$ (that corresponds to the hypotenuse). Higher profits hence correspond to shorter durations for the last fallow period, which is consistent with the previous remark. 
\end{itemize}

\begin{figure}[hpt]
   \centering
    \includegraphics[width=.55\linewidth]{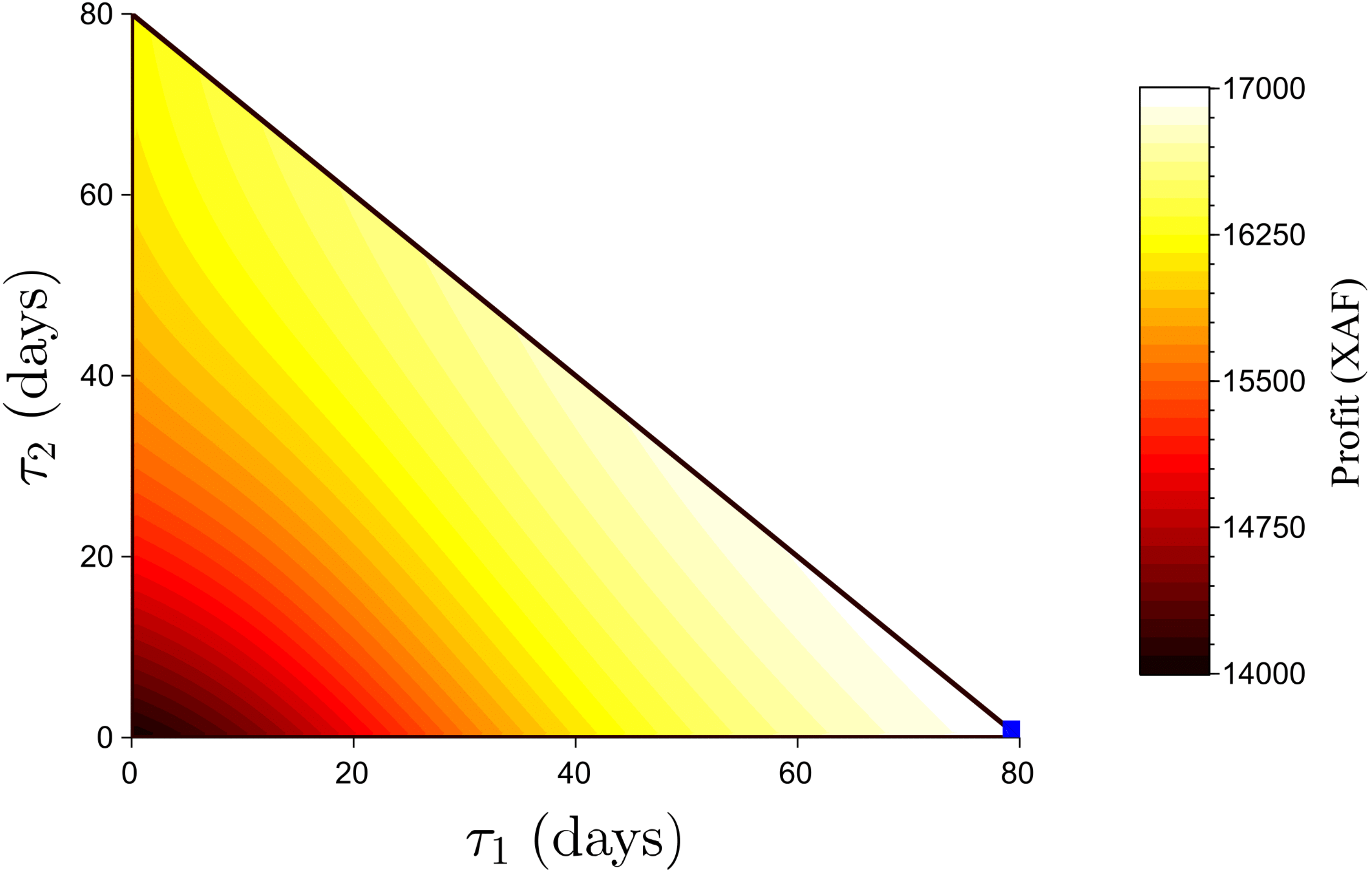}
      \caption{Profit as a function of the fallow period distribution on the $A^3$ simplex of size 80 days ($T_{max} = 1400$~days). The simplex is projected on its first two coordinates $(\tau_1, \tau_2)$ and $\tau_3=80-(\tau_1+\tau_2)$. The lighter the colour, the higher the profit. The maximum is indicated by a blue square and  corresponds to  $\tau_1=80$~days and $\tau_2=\tau_3=0$ day.}
    \label{Levels}
\end{figure}  

The strategy may be very different with a different time horizon, which nevertheless admits the same number of deployable seasons. For example, let $T_{max} = 1340$~days instead of 1400~days. In this case, it is preferable to deploy 3 cropping seasons (i.e.\ 2~fallow periods) and the optimal deployment is given by
$(\tau_1, \tau_2) = (332, 18)$ days. However, a 332-days fallow period is somehow too long, so we introduce an upper bound of 60~days  on each fallow period. This brings  the optimal solution back to 3 fallow periods, illustrated in Figure~\ref{Level20}. This figure shows that the maximum is located at the summit $\tau_3=20$~days. Since the total fallow duration ($\tau_1 + \tau_2 + \tau_3 = 20$~days) is small, it may be better to deploy it when the pest infestation is at its highest in order to maximise the fallow impact. In this case, a first 20-day fallow period is not long enough to sufficiently reduce the pest population, so it is more efficient to allocate these 20 days to the last fallow.

\begin{figure}[hpt]
   \centering
    \includegraphics[width=.55\linewidth]{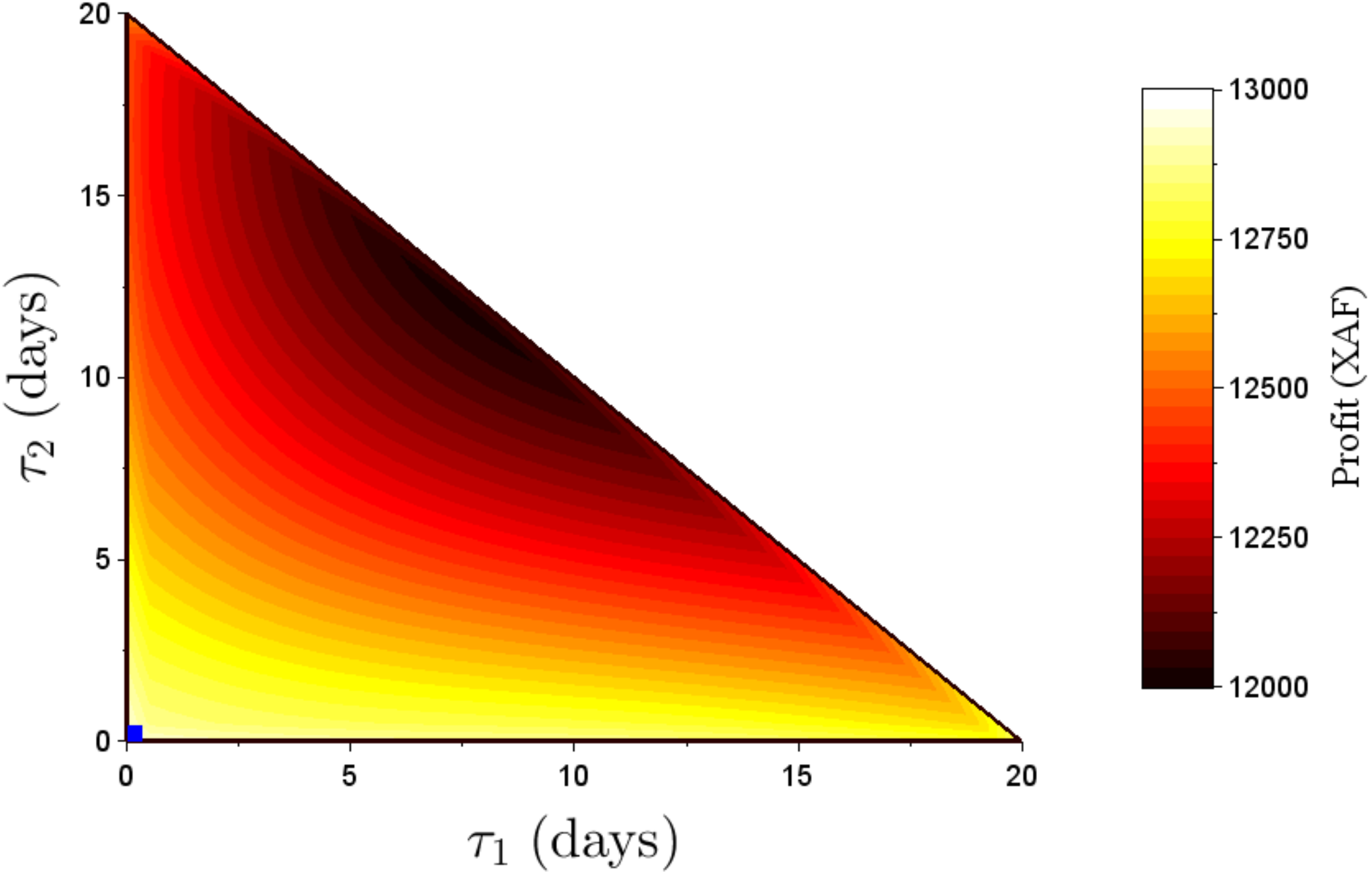}
      \caption{Profit as a function of the fallow period distribution on the $A^3$ simplex of size 20 days ($T_{max} = 1340$~days). The simplex is projected on its first two coordinates $(\tau_1, \tau_2)$ and $\tau_3=20-(\tau_1+\tau_2)$. The lighter the colour, the higher the profit. The maximum is indicated by a blue square and  corresponds to  $\tau_3=20$~days and $\tau_1=\tau_2=0$ day.}
    \label{Level20}
\end{figure}  

Still in the case of a time horizon of 1340~days, the optimum can be brought back from the summit $\tau_3 = 20$~days to the summit $\tau_1 = 20$~days when the initial infestation is large, and therefore the impact of the first fallow period is significant. We set $P_0 = 10000$ nematodes, instead of the reference value $P_0 = 100$ nematodes found in Table~\ref{Param}, and we illustrate the levels of infestation in Figure~\ref{Level20_10000}. As in Figure~\ref{Levels}, the maximum consists in deploying  the total fallow duration during the first period. In this case though, this strategy does not drastically reduce the pest population, but prevents it from increasing too much. Profits are globally lower than in the previous cases, whatever the fallow distribution.

\begin{figure}[hpt]
   \centering
    \includegraphics[width=.55\linewidth]{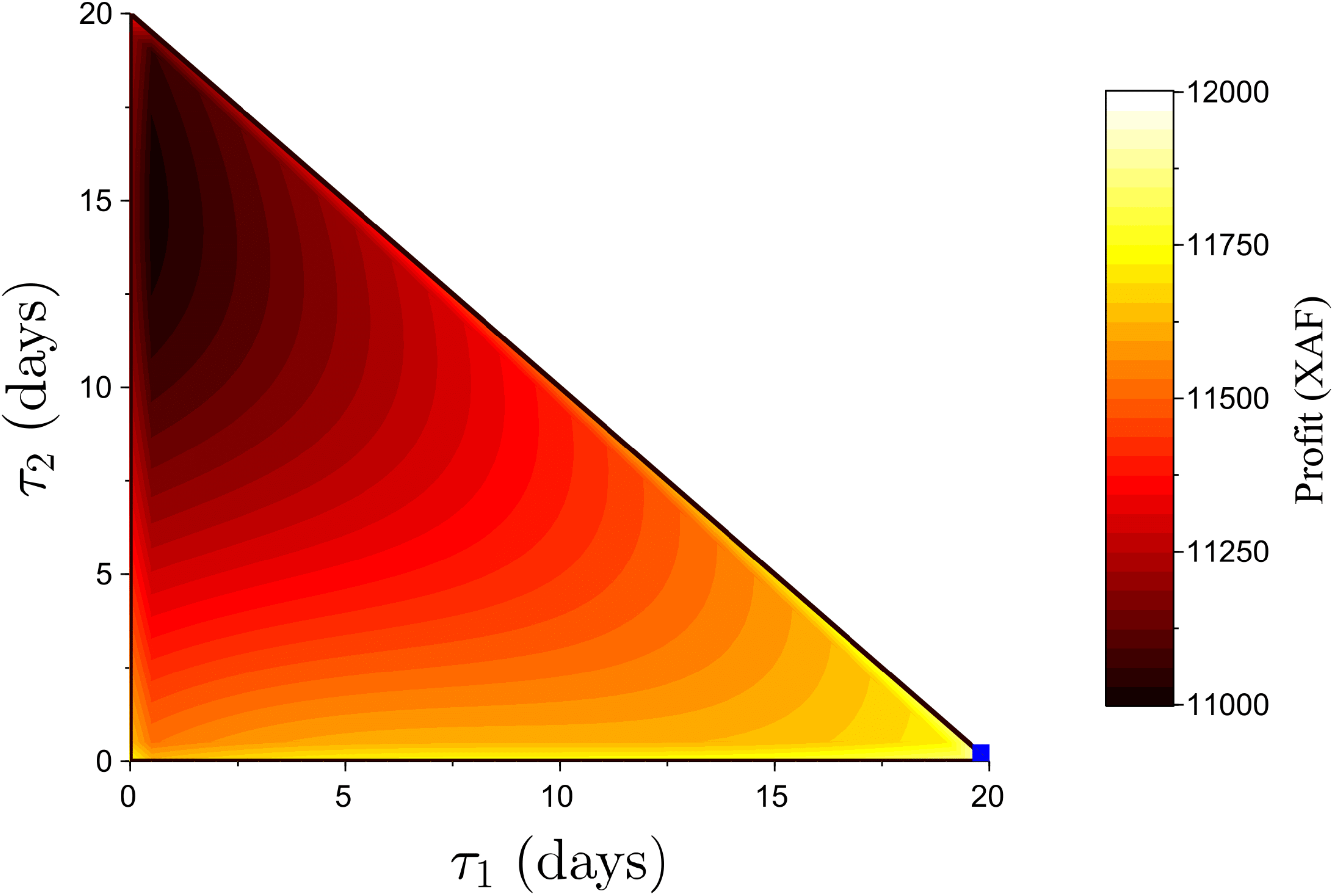}
      \caption{Profit as a function of the fallow period distribution on the $A^3$ simplex of size 20 days ($T_{max} = 1340$~days), for a very high initial infestation ($P_0=10000$ nematodes). The simplex is projected on its first two coordinates $(\tau_1, \tau_2)$ and $\tau_3=20-(\tau_1+\tau_2)$. The lighter the colour, the higher the profit. The maximum is indicated by a blue square and  corresponds to  $\tau_1=20$~days and $\tau_2=\tau_3=0$ day.}
    \label{Level20_10000}
\end{figure}

\subsection{High dimensions}\label{Subsection_high_dimension}

In high dimensions, i.e.\ when $T_{max}$ is large, we cannot easily illustrate the profit as a function of the fallow distribution. Moreover, thoroughly exploring the space of all the possible sequences of fallow periods would require a great deal of computation. Therefore, we solve the optimization Problem~\ref{probleme} using the Algorithm~\ref{algo} in subsection~\ref{Subsection_algorithm}.

We still use parameter values in Table~\ref{Param}. We set $T_{max}=4000$~days. Up to 12 cropping seasons, i.e.\ 11 fallow periods, can be deployed over this time horizon. However, the optimal deployment is obtained for 11 cropping seasons, which corresponds to a total of 370 days of fallow. It is illustrated in Figure~\ref{solution}. The corresponding optimal profit is $R(\vec{\tau}^*) =  54530$ XAF (83 euros) and the final soil infestation after the last harvest is $P(T_{max}^+) = 251$ nematodes.

\begin{figure}[htp]
   \centering
     \includegraphics[width=.6\linewidth]{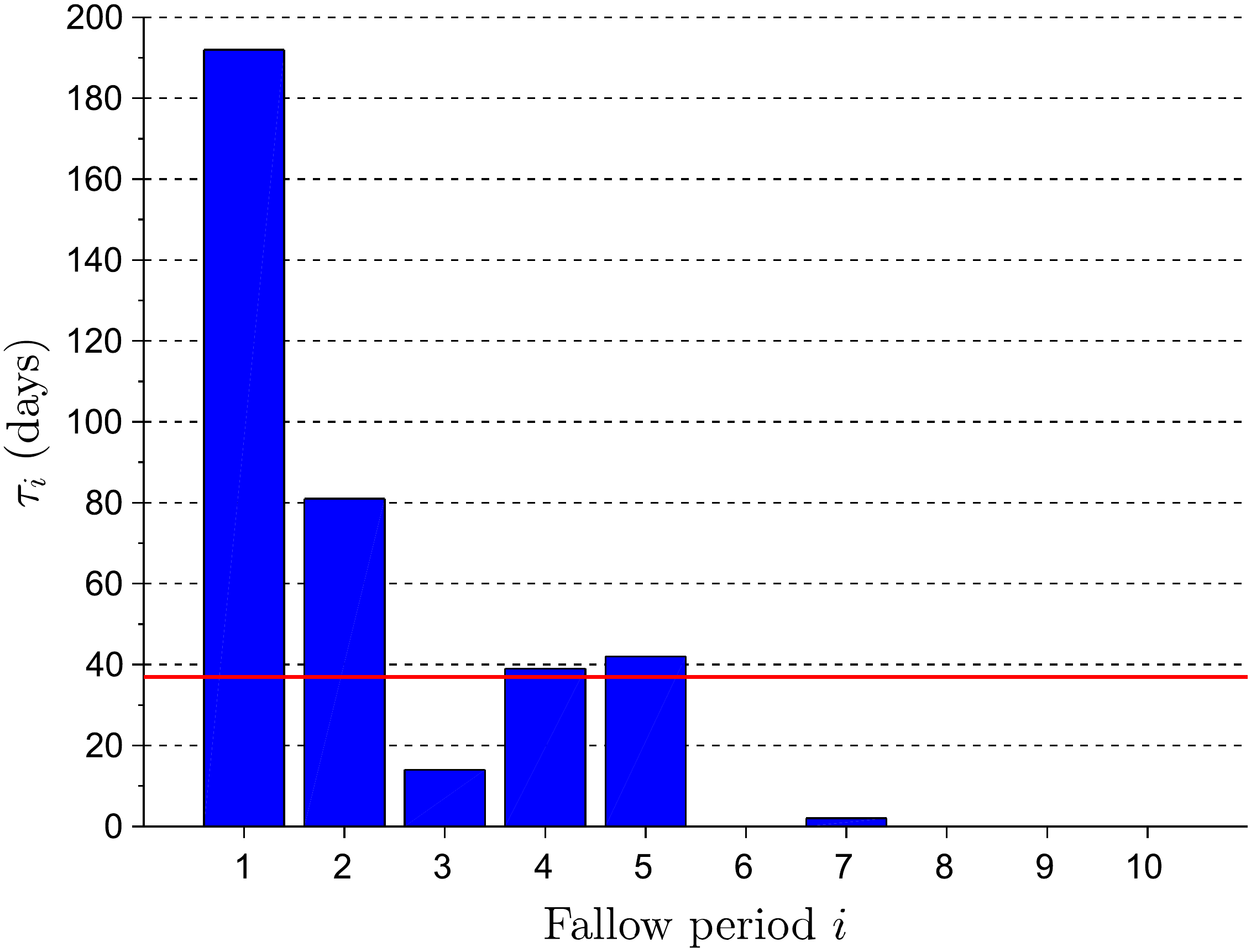}
    \caption{Optimal distribution of fallow periods for time horizon $T_{max}=4000$~days. The maximal profit is obtained for 11 cropping seasons and 10 fallows: $\vec{\tau}^*=(192,81,14,39,42,0,2,0,0,0)$ (in days). The red line corresponds to the average fallow period $\tau=37$~days.
    }
    \label{solution}
\end{figure}

\subsection{Regulation of high dimension solutions}\label{Section_regulation}

The optimal distribution of fallow periods found in Subsection~\ref{Subsection_high_dimension} and Figure~\ref{solution} is very dispersed around the average fallow duration. The first fallow period is huge whereas some others are null. Even if the strategy is optimal, farmers could be reluctant to implement such a an irregular cropping strategy. Besides, the level of infestation (251 nematodes) after the last harvest is somehow high, which would be a problem if the grower then cropped  a good host for \textit{R. similis}. It is necessary to find a compromise between the balance of the fallow durations and the profit.
In this subsection, we aim at limiting the durations of the fallow periods without penalizing the profit too much. First, we regulate the solution by bounding the duration of fallow periods. Then, we favour fallow periods that are close to the average  duration (that depends on the number of cropping seasons deployed). Finally, we consider constant fallow periods. We still use parameter values in Table~\ref{Param} and $T_{max} = 4000$~days for the numerical simulations.

\subsubsection{Bounded fallows}

The first regulation consists in bounding all fallow period durations $\tau_k$ by a maximal value $\tau_{sup}$. This means that $\tau_k \le \tau_{sup}$, $k = 1 \dots n$.
Since $\sum_{k=1}^n \tau_k =  T_{max}-(n+1)D$, we should have $n.\tau_{sup} \ge T_{max} -(n+1)D$. Hence:

\begin{equation*}
n \ge \frac{T_{max}-D}{\tau_{sup}+D}.
\end{equation*}
The optimal fallow distribution $\vec{\tau}^*$ should then be sought for dimensions $n$ between:

\begin{equation*}
  n_{min} = \left\lceil \frac{T_{max}-D}{\tau_{sup}+D} \right\rceil \text{\quad and \quad} n_{max} = \left\lfloor\frac{T_{max}}{D}\right\rfloor -1
\end{equation*}

We run Algorithm)~\ref{algo} for such dimensions and compare the profits obtained. When a $\tau_k$, chosen randomly, is greater than $\tau_{sup}$ in the ARS algorithm (\ref{ars}), we simply discard it and draw another one.

With the parameter values in Table~\ref{Param}, $T_{max} = 4000$~days and a maximal fallow duration $\tau_{sup} = 60$~days, the algorithm converges to the solution illustrated in Figure~\ref{Born_triv}. The associated  profit is $R(\vec{\tau}^*) =  54285$ XAF, which is just $0.4 \%$ worse than the non-regulated solution obtained in Subsection~\ref{Subsection_high_dimension}. The final soil infestation after the last harvest is $P(T_{max}^+) = 223$ nematodes.

\begin{figure}[htp]
   \centering
     \includegraphics[width=.6\linewidth]{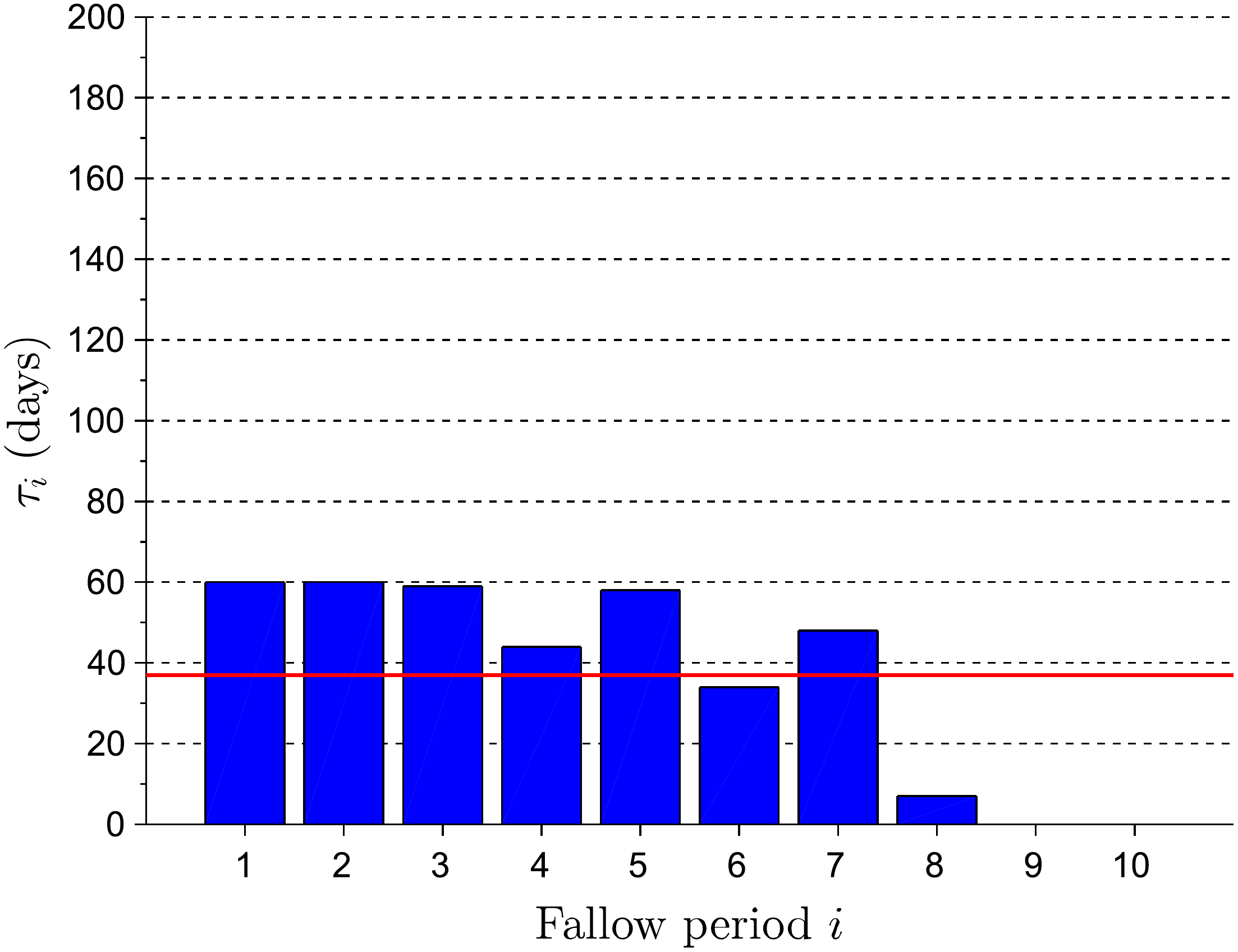}
    \caption{Optimal distribution of fallow periods for time horizon $T_{max}=4000$~days, when fallow durations are upper-bounded by $\tau_{sup}=60$~days. The maximal profit is obtained for 11 cropping seasons and 10 fallows: $\vec{\tau}^*=(60,60,59,44,58,34,48,7,0,0)$ (in days). The red line corresponds to the average fallow period $\tau=37$~days.}
    \label{Born_triv}
\end{figure}

\subsubsection{Penalizing dispersed fallows}

The second regulation consists in limiting the dispersion of the fallow distribution $\vec{\tau}$ around the average fallow duration, i.e.\ the distance between $\vec{\tau}$ and the centre of the simplex denoted by $\vec{\tau}_0$. Thereby, we introduce a penalty function in the expression of the profit, which is proportional to this distance $ d(\vec{\tau},\vec{\tau}_0)$, and define the penalized profit by:

\begin{equation}\label{regulated_yield}
  \tilde R(\vec{\tau}) = R(\vec{\tau}) - r \, d(\vec{\tau},\vec{\tau}_0).
\end{equation}
The regulation term $r$ is taken such that the magnitude of the penalty term $r \, d(\vec{\tau},\vec{\tau}_0)$ is an acceptable fraction of the magnitude of the unpenalized profit $R(\vec{\tau})$. Choosing $1/10$ for this faction, we deduce the value of $r$ as follows:

\begin{equation*}
  r = \frac{R(\vec{\tau}_0)}{10 \times d_{max}},
\end{equation*}
where $d_{max}$ stands for the longer distance to the centre of the simplex.

We apply Algorithm~\ref{algo} for the profit function $\bar R$ given by equation~\eqref{regulated_yield}. The algorithm converges to the solution illustrated in Figure~\ref{Born_ad}, using parameter values in Table~\ref{Param} and $T_{max} = 4000$~days. The associated  penalized profit is $\tilde R(\vec{\tau}^*)=54250$ XAF, which is just $0.5 \%$ worse than the non-regulated optimum  obtained in Subsection~\ref{Subsection_high_dimension}. The final soil infestation after the last harvest is $P(T_{max}^+) = 223$ nematodes.

\begin{figure}[htp]
   \centering
     \includegraphics[width=.6\linewidth]{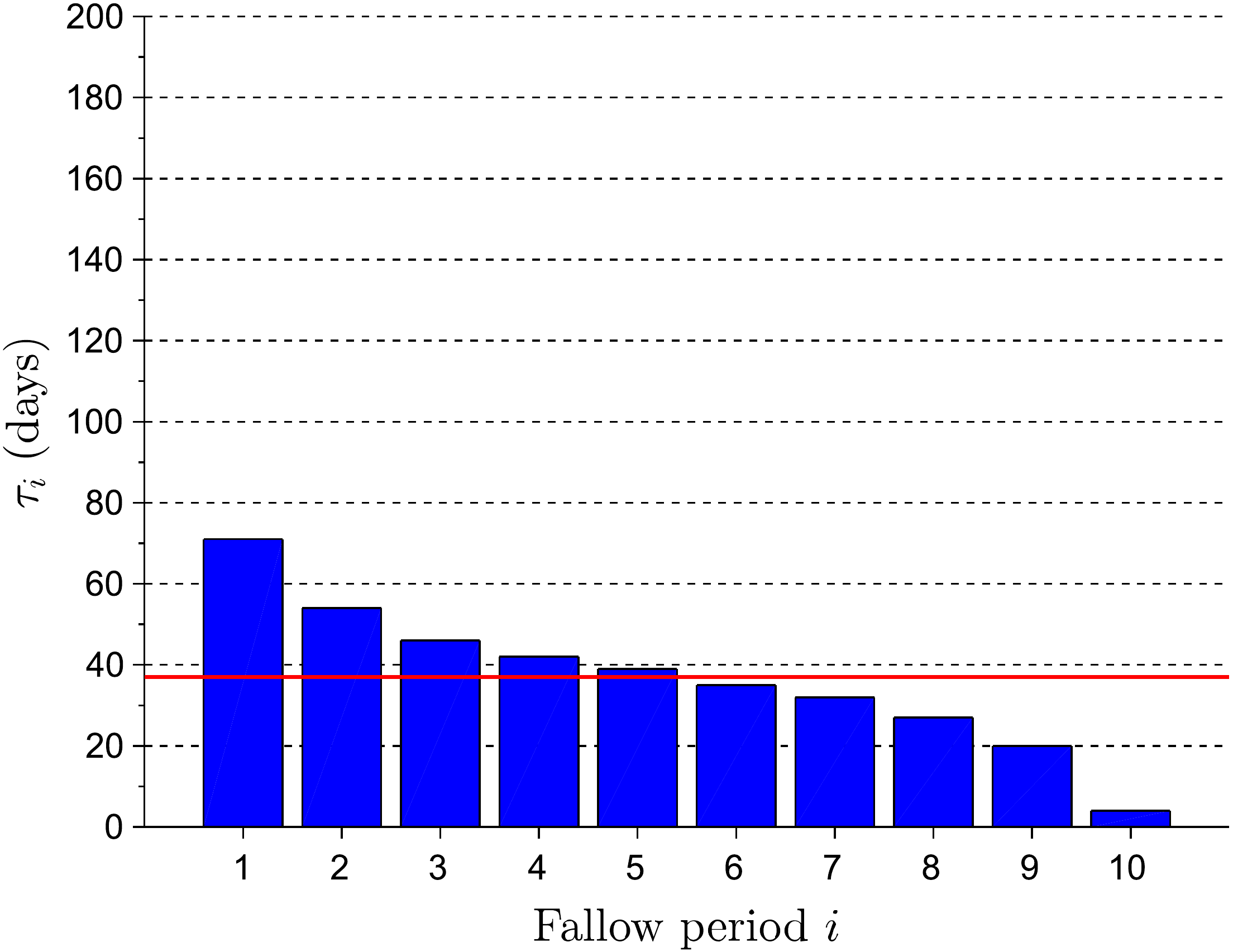}
    \caption{Optimal distribution of fallow periods for time horizon $T_{max}=4000$~days, when far from average fallow durationss are penalized. The maximal profit is obtained for 11 cropping seasons and 10 fallows: $\vec{\tau}^*=(71,54,46,42,39,35,32,27,20,4)$ (in days). The red line corresponds to the average fallow period $\tau=37$~days.
      }
    \label{Born_ad}
\end{figure}

\subsubsection{Constant fallows}

We previously deduced from monotonicity Assumption~\ref{assumption_mon} that the last harvest should occur at $T_{max}$ to optimize the profit. Indeed, since we could distribute the total fallow duration quite freely, it was always profitable to increase the fallow preceding the last cropping season, instead of deploying it at the end of the time horizon. In this section though, fallow durations are set to a constant value, which is an additional constraint that does allow the previous reasoning. We show below that, under the same Assumption~\ref{assumption_mon}, the last harvest of the optimal solution also occurs at $T_{max}$.

Given a fallow duration $\tau$, the  number of complete cropping seasons that can be deployed over time horizon $T_{max}$ is given by:

\begin{equation}\label{number_of_seasons}
  N \equiv N(\tau) = sup \{n \in \N \ | \ T_{max}  \ge nD + (n-1)\tau \},
\end{equation} 
As an incomplete cropping season yields no income, because of the sucker cost it induces a negative profit. Therefore, we assume that if $\tau$ leads to an incomplete season at the end end of the time horizon (corresponding to case~(a) in Figure~\ref{occurence} but with constant fallows), the last sucker is not planted. We can then rewrite the profit~(\ref{profit}) over $T_{max}$ as:

\begin{equation}\label{rendement_t_max}
  R(\tau)=\sum_{k=1}^{N(\tau)} R_k,
\end{equation}
and formulate the following optimization problem: 

\begin{equation}\label{constant_problem}
  \tau^* = \argmax_{\tau \ge 0} R(\tau).
\end{equation}

\begin{remark}\label{R_non_continue}
  Function $R(\tau)$ in equation~(\ref{rendement_t_max}) is not necessarily a continuous function of $\tau$. A discontinuity may occur for $\tau$ values such that $N(\tau+\varepsilon)=N(\tau)-1<N(\tau)$ for small positive values of $\varepsilon$.
  As an incomplete cropping season is not  profitable and hence not implemented, this small increase of the fallow duration wastes a whole cropping season.
\end{remark}

The previous remark shows that $\tau$ values such that $N(\tau + \varepsilon) < N(\tau)$ for  small positive values of $\varepsilon$, locally maximize the profit ``on the right'': $R(\tau) > R(\tau+\varepsilon)$, provided that the yield of a cropping season is higher than the cost of a sucker. This assumption is reasonable as it ensures the viability of the cropping system. If it did not hold, the profit would be negative and such $\tau$ values would minimize the profit ``on the right''.

Besides, if Assumption~\ref{assumption_mon} holds, then such $\tau$ values maximize the profit ``on the left''. Indeed, when two different fallow durations correspond to the same number of cropping seasons, the longer fallow leads to a greater reduction of the pest population during the fallow, that in turn leads, by monotonicity, to a greater root biomass during the following cropping season  and therefore to a better yield. Therefore, if Assumption~\ref{assumption_mon} holds, the solution of Problem~(\ref{constant_problem}) belongs to the set:

\begin{equation}\label{definition_xi}
\begin{aligned}
  \Xi = \left\{\tau \ge 0: \frac{T_{max} - D}{D+\tau} \in \N \right\} &= \left\{\frac{T_{max}-(n_{max}+1)D}{n_{max}},\ldots, T_{max}-2D \right\}, \\
 &\text{ with } n_{max} = \left\lfloor\frac{T_{max}}{D}\right\rfloor-1.
 \end{aligned}
\end{equation}
 
Still using parameter values in Table~\ref{Param} and $T_{max} = 4000$~days, we plot in Figure~\ref{Tau_singulier_Figure} profit $R$ as a function of fallow duration $\tau$. The maximizer $\tau^*=37$~days, leading to 10 fallow periods and 11 cropping seasons,  belongs to $\Xi$ as surmised. It corresponds to the average fallow period represented in Figures~\ref{solution}--\ref{Born_ad} (red line). The associated profit is $R(37)= 52000$~XAF, which is just $4.6\%$ worse than the non-regulated optimum obtained in Subsection~\ref{Subsection_high_dimension}. The final soil infestation after the last harvest is $P(T_{max}^+) = 82$ nematodes, which is much lower than for the non-regulated optimum. Figure~\ref{Tau_singulier_Figure} also shows that this optimal constant fallow is $54\%$ more profitable than no fallow ($R(0)=32150$~XAF). 

\begin{figure}[htp]
   \centering
    \includegraphics[width=.75\linewidth]{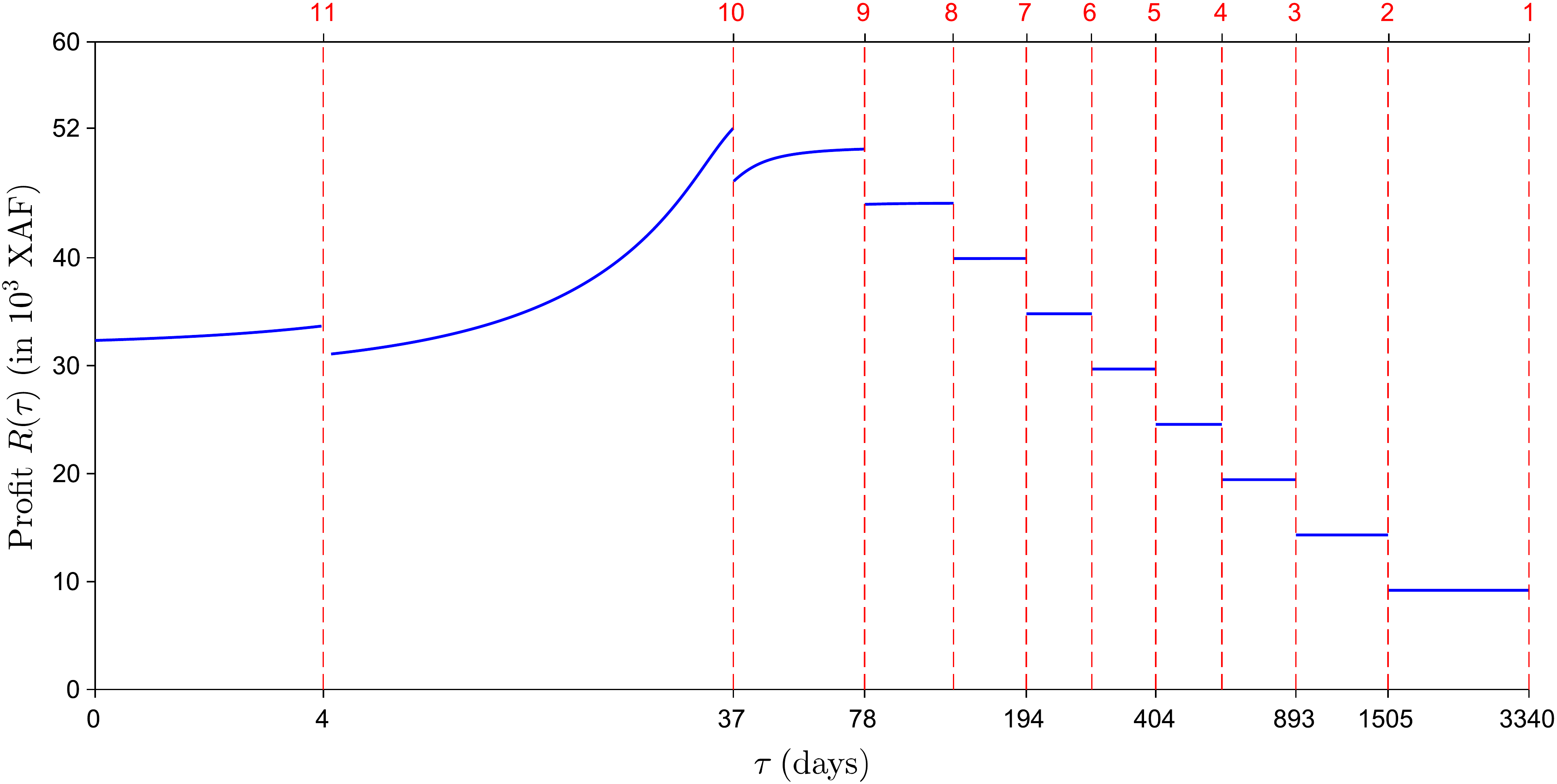}
    \caption{Profit $R$ as a function of fallow duration $\tau$ (logarithm scale) for time horizon $T_{max}=4000$~days. The set $\Xi=\{4,37,78,129,194,282,404,588,893,1505\}$ (in days), defined in equation~\eqref{definition_xi}, is represented by dashed red bars. Elements of $\Xi$ correspond to discontinuities of the profit function, when the number of fallows $n$ (upper axis) that fit in $T_{max}$ is incremented (from right to left). The maximal profit $R^*=52000$~XAF (79.27 euros) is obtained for $n=10$ fallows of duration $\tau^*=37$~days, which belongs to $\Xi$. }
    \label{Tau_singulier_Figure}
\end{figure}

Using the unrealistic parameters of the monotonicity counterexample in Figure~\ref{contreexemple}, for which Assumption~\ref{assumption_mon} is no longer valid, we can build a counterexample in which the optimal fallow period duration is not an element of set $\Xi$ defined above in equation~\eqref{definition_xi}. Indeed, setting $T_{max}=680$, only one fallow period can be deployed and $\tau=20$ is the only point of $\Xi$. However, as shown in Figure~\ref{contreexemple}(b), $\tau = 20$ does not maximize (nor minimize) the profit.

\subsection{Comparisons}

We compare the different optima obtained above when  $T_{max} =4000$~days, for the non-regulated and regulated strategies. In Figure~\ref{comp_inf}, we represent the  soil infestation after each harvest. In the most regular strategy, corresponding to constant fallows, the soil infestation follows a regular decrease over the seasons. For the other strategies, the soil infestation is first brought down, then rises again. The decrease and increase are sharper for the non-regulated strategy; in particular, the soil infestation is negligible right after the second harvest, but at its highest after the last harvest. The regulated strategies consisting in bounding or penalizing dispersed fallows induce similar soil infestations, especially after the last harvest at $T_{max}^+$. At this time, the soil infestation is much lower for constant fallows (up to three times lower).

\begin{figure}[htp]
   \centering
    \includegraphics[width=.75\linewidth]{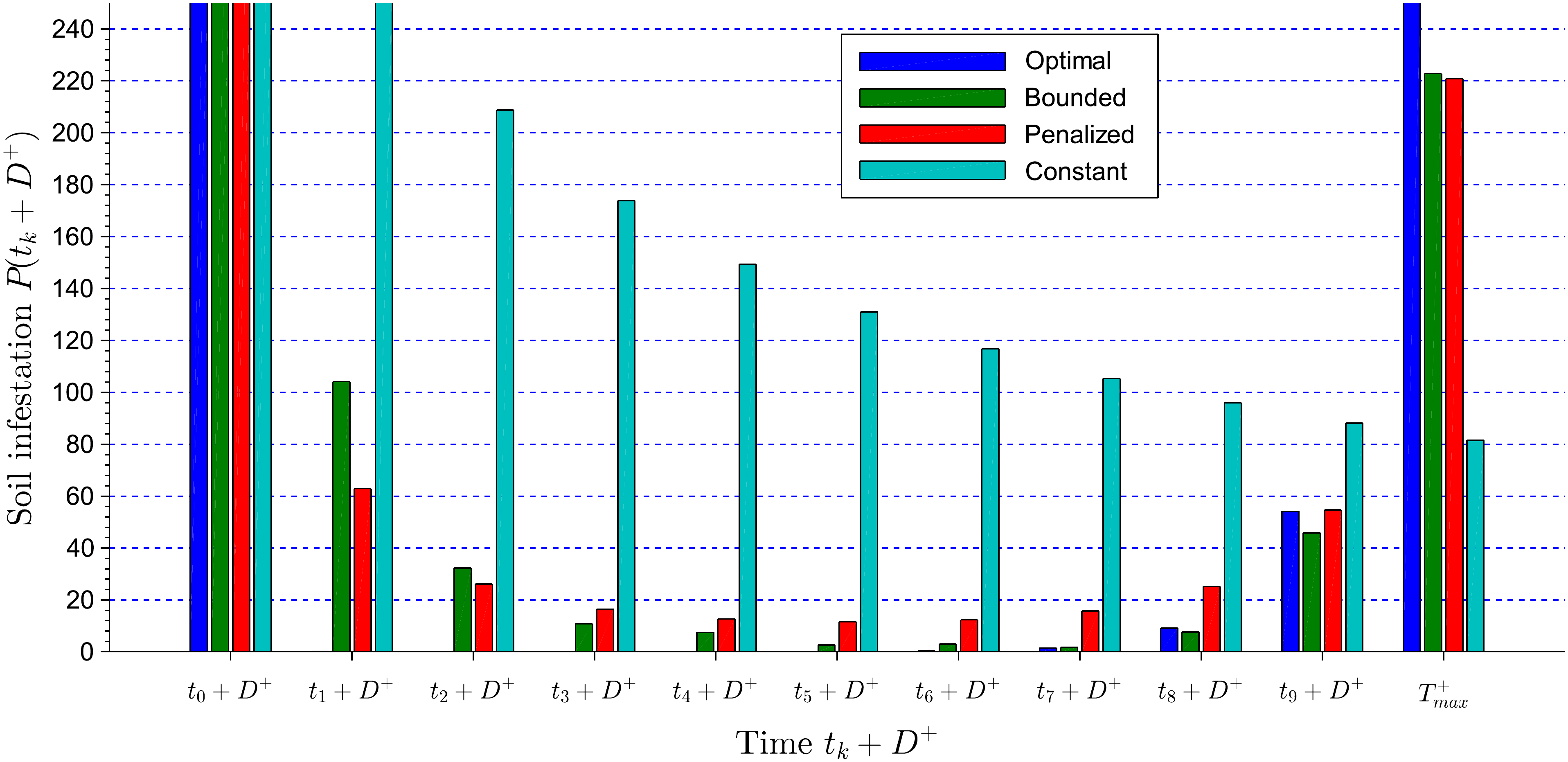}
    \caption{Soil infestation after each harvest for the optimal non-regulated and regulated fallow deployment strategies over time horizon $T_{max}=4000$~days. The non-regulated strategy (blue bars) corresponds to Figure~\ref{solution}. Regulations consist in bounding (green bars), penalizing (red bars) and setting constant (cyan bars) fallows; they correspond to Figures~\ref{Born_triv}, \ref{Born_ad} and \ref{Tau_singulier_Figure}, respectively. Soil infestations $P(t_0+D^+)$ after the first harvest are the same for all strategies, as initial conditions are the same. All strategies involve 10 fallows, but as their durations differ among strategies, times $t_k$ ($k=1,\dots,9$) also differ.}
    \label{comp_inf}
\end{figure} 
\begin{figure}[htp]
   \centering
    \includegraphics[width=.75\linewidth]{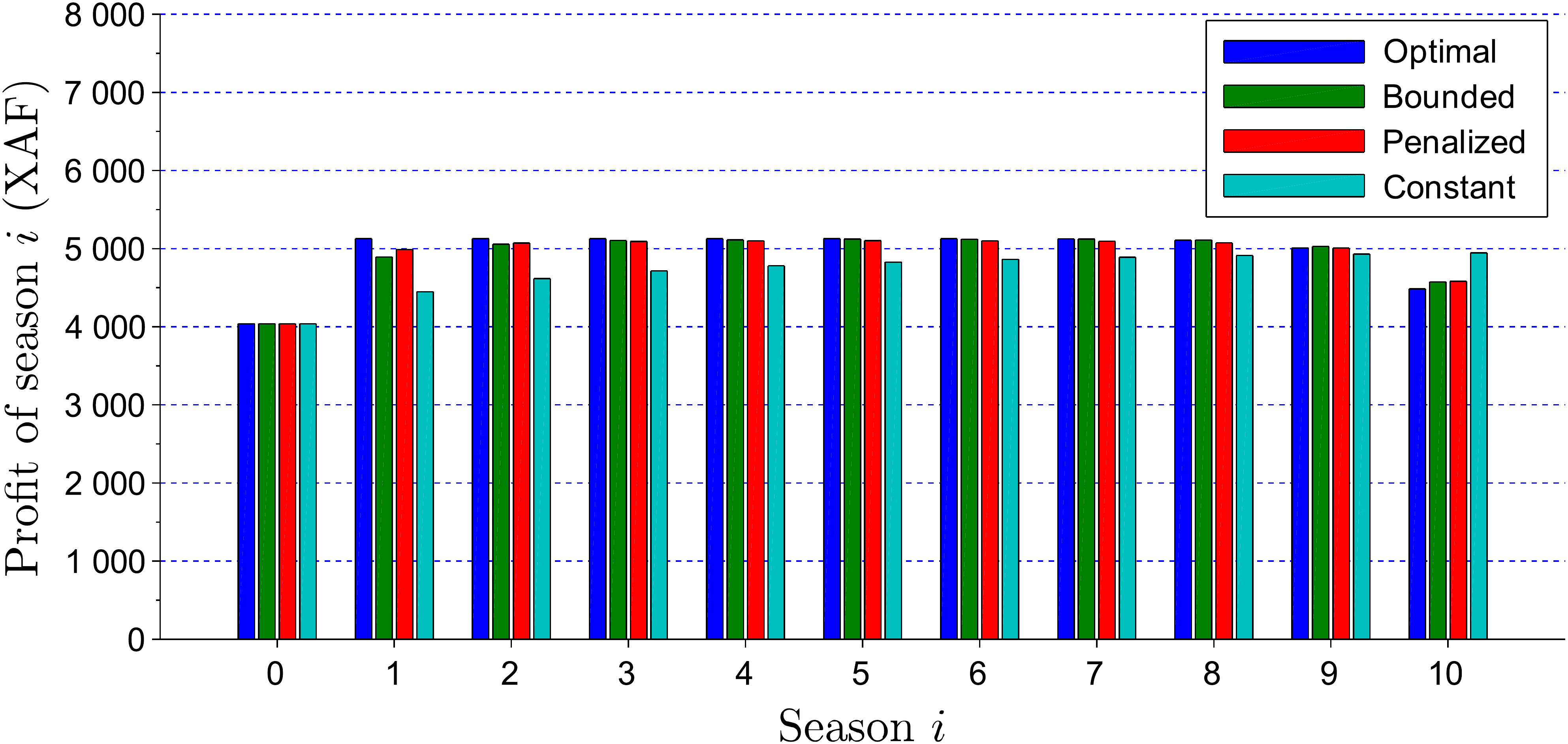}
    \caption{Seasonal profits of banana crop under different strategies of fallow deployment. for the optimal non-regulated and regulated fallow deployment strategies over time horizon $T_{max}=4000$~days. The non-regulated strategy (blue bars) corresponds to Figure~\ref{solution}. Regulations consist in bounding (green bars), penalizing (red bars) and setting constant (cyan bars) fallows; they correspond to Figures~\ref{Born_triv}, \ref{Born_ad} and \ref{Tau_singulier_Figure}, respectively. Profits of the first season are the same for all strategies, as initial conditions are the same.}
    \label{comp_prof}
\end{figure} 

The dynamical behaviour of soil infestation after each harvest is also reflected in the seasonal profits, since monotony (Assumption~\ref{assumption_mon}) makes lower infestations yield better profits. This is illustrated in Figure~\ref{comp_prof}. Seasonal profits vary much less than soil infestations. As shown above in Subsection~\ref{Section_regulation}, bounding or penalizing dispersed fallows yields total profits that very similar to the optimal with no regulation. This holds also for seasonal profits. With constant fallows, the seasonal profit increases regularly; at the last season, it is higher than the profits generated by the other strategies.

\clearpage
\section{Discussion and future work}\label{Section_discussion}

We have shown in this paper that increasing the duration of fallow periods tends to reduce the pest population. In an earlier work~\cite{TankamChedjou20}, we identified a threshold $\tau_0$ above which constant fallows lead to the disappearance of the pest asymptotically. With the  parameters in Table~\ref{Param}, this threshold is $\tau_0=36.8$~days. However, the systematic deployment of fallow periods longer than this threshold may not be optimal in terms of profit. On the one hand, such a deployment ensures that the pest declines in the long run, but in the short to medium term, quite longer fallows may be needed to significantly reduce the pest population and ensure higher seasonal profits. On the other hand, deploying long fallow periods could induce the loss of one or more cropping seasons on a given finite time horizon, which in turn could affect the total profit. Our optimization problem aimed at finding the right balance. For a time horizon of a little less than 11 years, we showed that it is preferable to deploy 11 rather than 12 cropping seasons in order to increase the total fallow time. The optimal solution consists in deploying a very long fallow after the first harvest, to drastically reduce the soil infestation, and then intermediate fallows during four more years (Figure~\ref{solution}). Pests remain relatively low until the end of the second to last cropping season, when they increase considerably (Figure~\ref{comp_inf}). The last seasonal profit hence decreases (Figure~\ref{comp_prof}), but further consequences of this optimal strategy would occur later, beyond the time horizon, which is a common issue for finite horizon optimization problems. In future work, to overcome this issue, we could penalize the final soil infestation.

In this work, we chose to tackle another issue exhibited by this optimal solution, which is  the dispersal of the fallow distribution around the average fallow duration (Figure~\ref{solution}). For several reasons, this solution may not be adopted by growers.. First, this solution implies an irregular crop calendar. The crop calendar is the schedule of cultural operations needed in crop production with respect to time, such as sowing, fertilising, harvesting. A regular schedule allows a  better planning of farm activities, including the distribution of labour. Second, another crop could be planted between banana cropping seasons instead of a fallow, provided that this inter-crop is a poor host of the pest. Otherwise, it would not help controlling the pest population. To implement such rotation, intervals between banana cropping seasons should long enough to grow the inter-crop. This is why the optimal solution was regularized by bounding fallows (Figure~\ref{Born_triv}), penalizing dispersed fallows (Figure~\ref{Born_ad}) or setting constant fallows (Figure~\ref{Tau_singulier_Figure}). This last regulation, besides being perfectly regular, leads to the lowest soil infestation after the last harvest, with only a small reduction of the total profit. Hence, the crops that are planted afterwards will benefit from a less infested soil. Constant fallows, possibly replaced by a poor host inter-crop, are therefore a good trade-off between profit and cultural constraints.

Determining such optimal fallow deployment strategies requires a good knowledge of the plant-pest interaction parameters, as well as the initial infestation. In this study, we gathered data from various published studies to inform our model parameters. Still, more quantitative experimental work on banana--nematode dynamics would help strengthen our conclusions.

There are two main limitations in our model that could lead to further developments. Firstly, we do not take into account the possible toppling of the plant. Indeed, above a certain damage level, the plant falls and the yield for that season is then totally lost \cite{Chabrier03, Chabrier05b}.  This  goes hand in hand with the monotonicity Assumption~\ref{assumption_mon} that ensures the ``good properties'' of our optimization problem. An infestation level high enough to induce the loss of monotonicity could lead to the toppling of the plant. Secondly, the use of nursery-bought healthy vitro-plants comes at a fairly high cost. Banana growers may prefer to rely on the vegetative reproduction of banana plants from lateral shots. This cultural practice does not allow for fallows between cropping seasons and is not very efficinet to control the soil infestation. A solution would be to alternate beween nursery-bought healthy suckers and vegetative reproduction. This  would lead to more complex optimal fallow deployment strategies.

\section*{Acknowledgement}
This work is supported by EPITAG, an Inria Associate team part of the LIRIMA (https://team.inria.fr/epitag/), and the EMS-Simons program for Africa. 

\bibliography{biblio}

\begin{appendix}
\section{Adaptive random search on the simplex}\label{ars}
  
The adaptive random search (ARS) algorithm consists in exploring a given bounded space, by alternating variance-selection and variance-exploitation phases \cite{Masri80,Walter97}. It is used here in Algorithm~\ref{algo} to solve maximization Problem~\ref{probleme}. It  is adapted to the simplex $A^n$ as follows. First, from a current point on the simplex, the displacement towards a new point of the simplex requires to randomly choose a direction  $\vec{d}=(d_k)_{k=1\dots n}$  such that $\sum_{k=1}^nd_k = 0$, and $||\vec{d}|| = 1$. Then, if the length of the displacement, drawn from a normal distribution $\mathcal{N}(0,\sigma)$,  is too large and such that the new point falls the limit of the simplex, this point is discarded and another displacement is drawn randomly.

The ARS algorithm, adapted to the $n$-simplex $A^n$, is described below. It aims at determining the optimal fallow distribution $\text{ARS}(n)=\vec\tau^{n,*}=(\tau_1^*,\dots,\tau_n^*)$ that maximizes the profit $R$ defined in equation~\eqref{profit} for a given number of fallows~$n$.

\subsubsection*{(Initialization)}
\begin{itemize}
\item[Step 1 --] Start as the center of the simplex:
  \begin{equation*}
    \vec\tau^{n,*} := \left[T_{max}-(n+1)D\right] \left(\frac{1}{n}, \dots,\frac{1}{n}\right)
  \end{equation*}
  and  initialize the standard deviation at the ``size'' of the simplex:\\ $\sigma^*= \sigma^0:=T_{max}-(n+1)D$. 
\end{itemize}

\subsubsection*{(Variance-selection)}
It aims at finding the best standard deviation $\sigma^*$.
\begin{itemize}
\item[Step 2 --] 5 decreasing standard deviations $\sigma^{i\in\{1,\dots,5\}} < \sigma^0 $ are chosen. For each standard deviation, $2\times n^2$ fallow distributions are drawn randomly in the simplex and their  profit is evaluated. The best standard deviation, selected for the next step, is the one corresponding to the highest profit. 
\end{itemize}
\begin{algorithm}[H]
  $\vec\tau_{sel} := \vec\tau^{n,*}$ \\
  \For{i := 1 to 5}{
    $\sigma^i := 0.3 \times \sigma^{i-1}$ \\
    \For{j := 1 to $2\times n^2$}{
      Draw $\vec{d}^{j}$  (\underline{cf.\ below})\\
      Draw $r^j\sim\mathcal{N}(0,\sigma^i)$\\
      $\vec\tau^j := \vec\tau_{sel} + r^j \vec{d}^j$ \\
      \While{$\vec\tau^j$ is outside of the simplex}{
        Draw $r^j\sim\mathcal{N}(0,\sigma^i)$ \\
        $\vec\tau^j := \vec\tau_{sel} + r^j \vec{d}^j$
      }
      \If{$R(\vec\tau^j) > R(\vec\tau^{n,*})$}{
        $\vec\tau^{n,*} := \vec\tau^j$ and $\sigma^* := \sigma^i$
      }
    }
  }
\end{algorithm}

\medskip
\underline{$\vec{d}^j$ draw}:
\begin{enumerate}
\item $\vec{d}^j \sim \mathcal{U} \big( [0,1]^n \big)$;
\item project $\vec{d}^j $ on the hyperplane $H = \{ (d_k) \in \R^n | \sum_{k=1}^n d_k = 0 \}$;
\item normalize $\vec{d}^j$.
\end{enumerate}

\subsubsection*{(Variance-exploitation)}
It aims at finding the best fallow distribution $\vec\tau^{n,*}$.
\begin{itemize}
\item[Step 3 --] 
$5\times n^2$ fallow distributions are drawn randomly in the simplex, using the best standard deviation $\sigma^*$ selected from the previous variance-selection phase, and their profit is evaluated. The best fallow distribution is the one with the highest profit.
\end{itemize}
\begin{algorithm}[H]
  \For{j := 1 to $5\times n^2$}{
    Draw $\vec{d}^{j}$ (\underline{cf.\ above})\\
    Draw $r^j\sim\mathcal{N}(0,\sigma^*)$\\
    $\vec\tau^j := \vec\tau^{n,*} + r^j \times \vec{d}$ \\  
    \While{$\vec\tau$ is outside of the simplex}{
      Draw $r^j\sim\mathcal{N}(0,\sigma^*)$\\
      $\vec\tau^j := \vec\tau^{n,*} + r^j \vec{d}^j$
    }
    \If{$R(\vec\tau^j) > R(\vec\tau^{n,*})$}{
      $\vec\tau^{n,*} := \vec\tau^j$
    } 
  }
\end{algorithm}

\subsubsection*{(Stopping criteria)}

Steps 2 and 3 are repeated until one of the following stopping criteria is achieved: 
\begin{itemize}
\item The smallest standard deviation $\sigma^5$ is used in more than 4 successive variance-exploitation phases.
\item The optimum is not improved in more than 4 successive variance-exploitation phases.
\item The profit is evaluated more than $100\times n^2$ times.
\end{itemize}

\end{appendix}

\end{document}